\documentclass[12pt,reqno,a4text]{amsart}
\usepackage{latexsym}
\usepackage{amsmath}
\usepackage{amssymb}
\usepackage[dvips]{graphicx}
\usepackage{latexsym}
\usepackage{enumerate,graphicx}
\usepackage[toc,page]{appendix}
\usepackage{setspace}
\usepackage{color,soul}
\usepackage{multirow}
\usepackage{a4wide}
\usepackage{amsthm}  
\usepackage{amsaddr}

\usepackage[colorlinks=true]{hyperref}

\newcommand{\R}{\mathbb{R}}

\newcommand{\X}{\mathbb {X}}

\def\be{\begin{equation}}
\def\ee{\end{equation}}

\def\wt{\widetilde}
\def\wh{\widehat}

\def\G{\Gamma}
\def\g{\gamma}

\def\p{\partial}

\def\om{\omega}
\def\mc{\mathcal}

\def\dsl{\displaystyle}

\def\ep{\epsilon}
\def\nb{\nabla}
\def\de{\delta}

\def\mR{\mathcal R}
\def\mP{\mathcal P}
\def\w{\wedge}

\def\ll{\lambda}

\newtheorem{theo}{Theorem}[section] 
\newtheorem{prop}[theo]{Proposition}

\theoremstyle{definition}
\newtheorem{exam}{Example}[section]

\newcommand{\fp}{\frak p}

\newcommand{\fc}{\frak c}

\begin{document}
\title{ Conserved quantities of distinguished curves on conformal sphere
}

\author{Prim Plansangkate$^1$$^\ast$ and Lenka Zalabov\' a$^2$}

\address{$^1$Division of Computational Science, Faculty of Science, Prince of Songkla University, Hat Yai, Songkhla, 90110 Thailand}
\email{prim.p@psu.ac.th}

\address{$^2$Department of Mathematics,
Faculty of Science, University of South Bohemia, \\
Brani\v sovsk\' a 1760, 370 05 \v Cesk\' e Bud\v ejovice, Czech Republic, and \\
Department of Mathematics and Statistics,
  Faculty of Science, Masaryk University, \\ Kotl\'{a}\v{r}sk\'{a} 2, 611 37 Brno,
Czech Republic }
\email{lzalabova@gmail.com}
\thanks{$^\ast$Corresponding author.}

\date{}


\vskip 50pt

\begin{abstract} 
We give conserved quantities of two generalizations of conformal circles on the conformal sphere.  One generalization concerns curves carrying a distinguished parallel tractor along them, which can be used to construct the conserved quantities.  The other is the class of curves satisfying the conformal Mercator equation, the conserved quantities of which are computed using Lagrangian formalism.  The two generalizations are not disjoint, and our main result is the relation between their conserved quantities.  Explicit examples are also presented.

\end{abstract}

\maketitle



\section{Introduction}

By distinguished curves on a conformal manifold, we mean a class of curves that is invariant under conformal transformations. A well-known example of these is the class of conformal circles (or conformal geodesics), whose definition could be traced back to Yano \cite{Y57}, who defined them as solutions of a third-order ODE.  Other equivalent definitions in terms of the defining ODEs are given in \cite{BE90, T12}.  There are also extensive studies of conformal circles based on tractor formulation, which can be found, for example, in \cite{BEG94, SZ19, EZ22} and references therein. 
In particular, in the key paper \cite{SZ19}, the authors give crucial observations on classes of conformal curves (including conformal circles) in the language of tractors and invariants.

This paper focuses on distinguished curves on the conformal sphere, which is the round $n$-sphere with the conformal class of the standard round metric.
The conformal sphere is a fundamental example of conformal manifolds.  In particular, it admits maximum conformal symmetries and is a homogeneous manifold.  Conformal circles on the conformal sphere are very well known; in fact, they are all round circles.  
Another example of distinguished classes of curves is the family of conformal loxodromes, recently defined by two different viewpoints in \cite{E23} and \cite{SZ19}. 

Our work here concentrates on two generalizations of conformal circles and the computation of their conserved quantities.
Recall that the conserved quantities for conformal circles have previously been studied, for example, in \cite{T12, SZ19, DT22, GST21}.  In \cite{T12, DT22}, conserved quantities were constructed from conformal Killing-Yano $2$-forms and then used to prove complete integrability of the conformal circle equation on gravitational instantons.  On the other hand, in \cite{GST21, SZ19} tractor calculus was employed.  In \cite{GST21} the authors constructed conserved quantities from tractors corresponding to conformal Killing-Yano 2-forms, while in \cite{SZ19}, with the assumption of projective parametrization, conserved quantities were constructed from conformal Killing vector fields.

Our first generalization follows \cite{GST21}.
Explicitly, they constructed a series of canonical tractors defined for each curve and showed that the wedge of the first three of them is parallel along the curve only if the curve is a conformal circle. Then, pairing it with suitable BGG solutions gives conserved quantities of conformal circles. We focus in Section \ref{par-trac} on curves such that the wedge of the first four canonical tractors is parallel and simultaneously discuss their invariants according to \cite{SZ19}. On the conformal round sphere, 
this allows us to
find conserved quantities for such curves by pairing with suitable parallel tractors.

The second generalization is based on a Lagrangian formulation for the conformal circle  equation.  It was shown in \cite{DK21} that the conformal circle equation, on an arbitrary conformal manifold, can be derived from a conformally invariant $3$rd order Lagrangian, provided that one considers a special class of variations.\footnote{Variational formulation of the conformal circle equation has also recently been studied in \cite{M24, KMS24, K25}.   It was shown that in $3$ dimensions the equation for unparametrized conformal circles is precisely the Euler-Lagrange equation of the total torsion functional \cite{M24, KMS24}, while the equation for parametrized conformal circles is not variational in the classical sense \cite{KMS24}.  Later, in \cite{K25}, it was proved that the conformal circle equation, either for unparametrized or parametrized curves, is not variational in any dimension greater than $3.$} 
The Euler-Lagrange equation of this Lagrangian, on the other hand, is of $4$th order, and called the conformal Mercator equation in \cite{DK21}.  In fact, the equation also arises from the $2$nd order Lagrangian discussed in \cite{BE90}, which differs from the $3$rd order Lagrangian of \cite{DK21} by a total derivative term.  It was noted \cite{BE90, DK21} that on a conformally flat manifold such as the conformal sphere, the solutions of the conformal Mercator equation include conformal circles.  Thus, we shall consider parametrized curves which are solutions of the conformal Mercator equation as another class of distinguished curves generalizing the conformal circles.\footnote{Relation between solutions of the conformal Mercator equation and conformal loxodromes has also been discussed in \cite{E23}. }
As the associated Lagrangian is conformally invariant, we are able to use Noether's principle to derive conserved quantities for the solution curves from conformal Killing vector fields.  This result is given in Section \ref{sec:firstintegralF}.

Via an explicit example, it is evident that the two generalizations we consider are not disjoint, that is, they share some common curves on which the two sets of our conserved quantities are both constant. This motivates us to find a relation between the generally different quantities.  Using the Hamiltonian formulation, which simplifies the expressions of the quantities, we establish such a relation.  This is our main result and is given in Section \ref{sec:Hamiltonian}.


\section{Invariants via parallel tractors} \label{par-trac}
We use here tractor calculus to study parallel tractors and conserved quantities of curves on the conformal sphere. One can see e.g. \cite{CG17} for an introduction to conformal tractors. 
\subsection{Conformal sphere, tractors and coordinates}
By a conformal sphere we mean the sphere $S^n$ with the conformal class of metrics represented by the standard round metric. It is a homogeneous space and in fact, maximally symmetric conformal Riemannian structure of dimension $n$. The Lie group of conformal transformations of $S^n$ is isomorphic to $O(n+1,1),$ with the corresponding Lie algebra $\frak{so}(n+1,1).$ 
 This follows from the realization of $S^n$ as the projectivization of the cone of non-zero null-vectors in the pseudo-Euclidean space $\R^{n+1,1}$ of signature $(n+1,1)$.
The group $O(n+1,1)$ acts transitively on $S^n$ and the stabilizer of a point is the Poincar\'e subgroup $P\subset O(n+1,1)$, so $S^n\cong O(n+1,1)/P$.

We realize the standard scalar product of signature $(n+1,1)$ in $\R^{n+1,1}$ by the block $(1,n,1)$-matrix
\begin{align} \label{met}
G:=\left( \begin{smallmatrix}
0&0&1 \\ 0& E &0 \\ 1 &0&0
 \end{smallmatrix} \right)
\end{align}
for the identity matrix $E$. Then we 
write elements of $\frak{so}(n+1,1)$ as block $(1,n,1)$-matrices
$$
\left( \begin{smallmatrix} -a&2S&0 \\ T &-R&-2S^t \\ 0&-T^t&a \end{smallmatrix} \right)
$$
for $a \in \R$, $R \in \frak{so}(n)$ and $T,S^t \in \R^n$. 
Viewed as conformal Killing vector fields, these are
$$(T^i-{R^i}_j x^j+ax^i-2S_jx^jx^i+x_jx^jS^i)\frac{\partial}{\partial x^i}.$$
The Lie algebra $\frak{p} \subset \frak{so}(n+1,1)$ corresponds to the upper triangular matrices
and we fix the complement $\fc$ of $\frak{p}$ in $\frak{so}(n+1,1)$ such that $\frak{so}(n+1,1)=\fc \oplus \fp$ as follows: 
$$
\fp \ni \left( \begin{smallmatrix} -a&2S&0 \\ 0 &-R&-2S^t \\ 0&0&a \end{smallmatrix} \right), \ \ 
\fc \ni \left( \begin{smallmatrix} 0&0&0 \\ T &0&0 \\ 0&-T^t&0 \end{smallmatrix} \right).
$$
The complement $\fc$ gives local exponential coordinates on $S^n$ around the origin $o=eP$, which correspond to the usual coordinates coming from the projectivization of the cone in $\R^{n+1,1}$. Indeed, we get
$$
\exp \left( \begin{smallmatrix} 0&0&0 \\ T &0&0 \\ 0&-T^t&0 \end{smallmatrix} \right) \left( \begin{smallmatrix} 1 \\ 0 \\ 0 \end{smallmatrix} \right) = \left( \begin{smallmatrix} 1&0&0 \\ T &E&0 \\ -\frac12\langle T,T \rangle &-T^t&1 \end{smallmatrix} \right)  \left( \begin{smallmatrix} 1 \\ 0 \\ 0 \end{smallmatrix} \right) = \left( \begin{smallmatrix} 1 \\ T\\ -\frac12 \langle T,T \rangle \end{smallmatrix} \right),
$$
where $\langle\: , \: \rangle$ is the standard scalar product on $\R^n$.

We employ the tractor point of view in Section \ref{par-trac}, see \cite{CG17} for an introduction to tractor calculus. Due to the homogeneity of the conformal sphere, we follow \cite{GZ} to describe standard tractors in coordinates. We consider the standard tractor bundle $\mathcal{T}$ modelled on $\R^{n+2}$ with the tractor metric, which is the standard pseudometric $G$ of signature $(n+1,1)$. Thus, we deal with the standard representation $$\rho: \frak{so}(n+1,1) \to \frak{gl}(\R^{n+2}),$$ 
and we employ the description of the tractors in the exponential coordinates just by restriction to $\fc$. 
Indeed, the standard tractor bundle $\mathcal{T}$ then trivializes as $\mathcal{T} \simeq \fc \times \R^{n+2}$ and we can see the standard tractors $\mathbb{X} \in \mathcal{T}$ as sections of $\fc \times \R^{n+2}$, that is, functions $ \fc \to \R^{n+2}$. 
Then, written in this local trivialization, the corresponding functions $\mathbb{X}$ take the form
\begin{gather} \label{exp}
\exp(J) \mapsto \exp(-\rho(J))(w), \ \ \ \ J=\left( \begin{smallmatrix} 0&0&0 \\ X &0&0 \\ 0&-X^t&0 \end{smallmatrix} \right)
\end{gather}
for some neighborhood $X$ of $0$ in $\fc$ and $w=(w^a)$ are coordinates in the representation space $\R^{n+2}$.  Viewing $\rho(J)$ as a linear transformation, we get
\begin{gather} \label{action2}
e_0 \mapsto \sum_{i=1}^n X^ie_i, \ \ \ \  e_i \mapsto -X^ie_{N}, \ \ \ \   e_N \mapsto 0,
\end{gather}
where $X=(X^i)$, $i=1, \dots, n$ and $N:=n+1$.
Thus, in coordinates $w$, we get
\begin{gather} \label{action-together}
\rho(J)(w)=w^0 \sum_{i=1}^n X^ie_i-\sum_{i=1}^n w^i X^i e_N.
\end{gather}
 One can then compute directly
 \begin{gather} \label{expw}
 \exp(-\rho(J))(w) = w-\rho(J)(w)+\frac12  \rho^2(J)(w),
 \end{gather}
 and then the series stops as $\rho^3(J)$ vanishes. 
The same ideas apply analogously to arbitrary tractor bundles just by tensoriality; we are particularly interested in $\wedge^4\R^{n+2}$.

Let us finally say that the standard tractor connection decomposes into the fundamental derivative $D$ and the algebraic action, and in the case of the conformal sphere takes the form 
$\nabla^\mathcal{T}=D+\rho$, restricted to $\fc$ to get coordinate description.
We then get the tractor connection on an arbitrary tensor bundle analogously by taking the corresponding representation of $\frak{so}(n+1,1)$.

\subsection{Curves, invariants and parallel tractors} \label{sec:curvetractor}
We consider an arbitrary curve $\gamma=(\gamma^i(t))$, $i=1, \dots, n$, and we denote 
$$U=(\frac{d}{dt}\gamma(t)), \ \ A=(\frac{d^2}{dt^2}\gamma(t)), \ \ A'=(\frac{d^3}{dt^3}\gamma(t)), \dots$$
its velocity vector field, acceleration vector field, and so on for subsequent derivatives, where $\,{}'$ denotes $\frac{d}{dt}.$
We denote by $u$ the length of the velocity vector field
$u=\sqrt{\langle U,U \rangle}$
for the standard scalar product on $\R^n$.

For each curve $\gamma$, we define a canonical series of derived standard tractors by means of the tractor derivative. Unlike the usual way of writing standard tractors in columns, \cite{CG17,SZ19,GST21}, we write them 
in coordinates in the standard basis $e_a,$ $a=0, \dots, N,$ of $\R^{n+2}$ that we decompose (for the canonical metric from the conformal class) as 
$$ \R^{1+n+1}=\langle e_0 \rangle + \langle e_1, \dots, e_n \rangle + \langle e_{N} \rangle.
$$
Let us note that this notation is useful for dealing with higher powers of the representation $\rho$ from \eqref{action-together} we need in Section \ref{sec:firstintegralF}. 
We start the series of standard tractors with the canonical inclusion 
\begin{gather} \label{trT}
\mathbb{T}= u^{-1}e_0.
\end{gather}
Then we compute along the curve $\gamma$ the tractor derivative $U^a\nabla^\mathcal{T}_a$, which takes on the tractor $\mathbb{X}$ the explicit form
\begin{align} \label{tra}
U^a\nabla^\mathcal{T}_a(\mathbb{X})=D(\mathbb{X})+ \rho(K)(\mathbb{X}),\ \ \ \  K:=\left( \begin{smallmatrix} 0&0&0 \\ U&0&0 \\ 0&-U&0 \end{smallmatrix} \right).
\end{align}
 Direct computation using \eqref{tra} leads to the sequence of standard tractors $\mathbb{U}= U^a\nabla^\mathcal{T}_a(\mathbb{T})$, $\mathbb{A}= U^a\nabla^\mathcal{T}_a(\mathbb{U})$, $\mathbb{A'}= U^a\nabla^\mathcal{T}_a(\mathbb{A})$ etc. 
 where we use Einstein summation notation and $i=1,\dots,n$
\begin{align}
\mathbb{U}&=-u^{-3}\langle U,A \rangle e_0+u^{-1} U^ie_i \label{trU}
\\
\mathbb{A}&= (3u^{-5}\langle U,A \rangle^2-u^{-3}(\langle A,A \rangle +\langle U,A'\rangle ))e_0+
(-2u^{-3} \langle U,A \rangle U^i+u^{-1}A^i)e_i-u e_{N} \label{trA}
\\
\begin{split}
\mathbb{A}'&= (-15u^{-7}\langle U,A \rangle^3+9u^{-5}\langle U,A \rangle(\langle A,A \rangle +\langle U,A'\rangle)
-u^{-3}(3\langle A,A' \rangle +\langle U,A'' \rangle) )e_0
\\ &+
(9u^{-5}\langle U,A \rangle^2U^i-3u^{-3}(\langle A,A \rangle +\langle U,A'\rangle)U^i-3u^{-3}\langle U,A \rangle A^i+u^{-1}(A')^i)e_i
\label{trAbar}
\end{split}
 \\
 &\textrm{etc.} \notag
\end{align}

For each (parametrized) curve, thus for each sequence of tractors $\mathbb{T}, \mathbb{U},  \mathbb{A}, \mathbb{A'}, \dots $ of arbitrary length $\ell$,
we compute the following Gram matrix of order $\ell$ 
 $$
M_\ell=\left( \begin{smallmatrix}
  0&0&-1&0&\alpha_1 & \cdots\\
   0&1&0&-\alpha_1&-\frac32 \alpha'_1
 & \cdots\\
   -1&0&\alpha_1&\frac12\alpha'_1&\frac12 \alpha''_1-\alpha_2& \cdots\\
  0&-\alpha_1&\frac12 \alpha'_1&\alpha_2&\frac12 \alpha'_2& \cdots\\
   \alpha_1&-\frac32 \alpha'_1&\frac12 \alpha''_1-\alpha_2&\frac12 \alpha'_2&\alpha_3 & \cdots \\
   \vdots &  \vdots & \vdots & \vdots & \vdots & \ddots
  \end{smallmatrix} \right)
$$
for the standard tractor metric, 
 compare to \cite{SZ19}. Thus $\alpha_1$ is the square of the length of the acceleration tractor  $\mathbb{A}$ and analogously for next tractors/alphas in the series.
In particular, we compute
\begin{align} \label{alpha}
\alpha_1 = -6u^{-4}\langle U,A \rangle^2+2u^{-2}\langle U, A' \rangle +3u^{-2}\langle A,A \rangle.
\end{align}
This also allows us to define a series of natural relative invariants $\Delta_\ell=\det(M_\ell)$, \cite[Lemma 2.2]{SZ19}. Let us note that $\Delta_1=\Delta_2=0$ and $\Delta_3=-1$. In particular, the tractors $\mathbb{T}$, $\mathbb{U}$ and $\mathbb{A}$ are linearly independent.
The first non-trivial invariant then is $\Delta_4= \alpha_1^2-\alpha_2$, and  
\begin{align}
\begin{split} \label{Delta}
\Delta_4&=
9u^{-8}\langle U,A \rangle^4-6u^{-6}\langle U,A \rangle^2\langle U,A'\rangle-9u^{-6}\langle U,A \rangle^2\langle A,A\rangle \\ 
&+6u^{-4}\langle U,A\rangle\langle A,A'\rangle +u^{-4}\langle U, A' \rangle^2-u^{-2}\langle A',A'\rangle.
\end{split}
\end{align}
Let us emphasize that $\Delta_4$ is non-positive, \cite[Lemma 2.3.]{SZ19}.

The best known and studied class of conformal curves are conformal circles, \cite{EZ22,GST21,SZ19}. In the language of tractors, a curve is a conformal circle if and only if tractors $\mathbb{T}, \mathbb{U}, \mathbb{A}, \mathbb{A'}$
are linearly dependent, i.e. $\mathbb{A'} \in \langle \mathbb{T}, \mathbb{U}, \mathbb{A} \rangle$. It then follows that the curve is a conformal circle if and only if $\Delta_4=0$ and we can use the Cramer rule for $M_4$ to write the tractor combination explicitly. Direct computation  gives 
$$ \mathbb{A'}=-\alpha_1\mathbb{U}-\frac12 \alpha'_1 \mathbb{T}.$$
This leads to the condition that $T_3=\mathbb{T} \wedge \mathbb{U} \wedge \mathbb{A}$ is parallel along $\gamma$ for the tractor connection,   \cite[Section 4.4.3]{GST21}.  
 Applying the tractor derivative of the tractor bundle  $\wedge^3 \mathcal{T}$ on $T_3$ gives
$$
\nabla^{\wedge^3 \mathcal{T}}(\mathbb{T} \wedge \mathbb{U} \wedge \mathbb{A})=\mathbb{U} \wedge \mathbb{U} \wedge \mathbb{A}+\mathbb{T} \wedge \mathbb{A} \wedge \mathbb{A}+\mathbb{T} \wedge \mathbb{U} \wedge (-\alpha_1\mathbb{U}-\frac12 \alpha'_1 \mathbb{T})=0.
$$

Analogous ideas generalise for an arbitrary number of tractors in the series. The authors prove in \cite[Section 2.3.]{SZ19} that: 
\begin{center}
vanishing of $\Delta_\ell$ $\Longleftrightarrow$ tractors $\mathbb{T}, \mathbb{U},\mathbb{A}, \mathbb{A'}, \dots, \mathbb{A}^{\ell - 3}$ are linearly dependent. 
\end{center}
Note also that vanishing of $\Delta_\ell$ implies vanishing of $\Delta_{\ell+1}$ and signs of $\Delta_\ell$ are independent of reparametrizations.

Let us now generalise to the case of $\Delta_5$. Thus we assume $\Delta_4 \neq 0$ and $\Delta_5=0$, which means that the tractors $\mathbb{T}, \mathbb{U},  \mathbb{A},  \mathbb{A}', \mathbb{A}''$ are linearly dependent, i.e. $ \mathbb{A}'' \in \langle \mathbb{T}, \mathbb{U},  \mathbb{A},  \mathbb{A}'\rangle$. Using $M_5$ and the Cramer rule, we find the tractor combination explicitly as
\begin{align} \begin{split}
\label{tract-comb-5}
\mathbb{A}''&=\frac{\Delta'_4}{2\Delta_4}\mathbb{A}'-\alpha_1\mathbb{A}+\frac{1}{2\Delta_4}(\alpha'_1(2\alpha_2-\Delta_4)-\alpha_1\alpha'_2)\mathbb{U}\\ &+
\frac{1}{4\Delta_4}(2\alpha_1(\alpha'_1)^2-4\Delta_4^2-2\alpha''_2\Delta_4-\alpha'_1\alpha'_2)\mathbb{T}.
\end{split}
\end{align}
 We are interested in analogues of the $3$-tractor $T_3$. 
Let us derive the expression $ \Delta_4^a \cdot \mathbb{T} \wedge \mathbb{U} \wedge \mathbb{A} \wedge \mathbb{A'}$  with respect to tractor connection along the curve, where we consider a general power of $\Delta_4$ due to the possible effect of the coefficient with $\mathbb{A}'$ in \eqref{tract-comb-5}. Applying the tractor derivative of the tractor bundle  $\wedge^4 \mathcal{T}$ on the expression and using \eqref{tract-comb-5} gives
\begin{align}
\begin{split}
\label{derivative-4-tractor}
&\nabla^{\wedge^4 \mathcal{T}} (\Delta_4^a \cdot \mathbb{T} \wedge \mathbb{U} \wedge \mathbb{A} \wedge \mathbb{A'})=
a\Delta_4^{a-1} \Delta_4' \cdot \mathbb{T} \wedge \mathbb{U} \wedge \mathbb{A} \wedge \mathbb{A'}+ \Delta_4^a \cdot \mathbb{U} \wedge \mathbb{U} \wedge \mathbb{A} \wedge \mathbb{A'}\\
&+\Delta_4^a \cdot \mathbb{T} \wedge \mathbb{A} \wedge \mathbb{A} \wedge \mathbb{A'} +
\Delta_4^a \cdot\mathbb{T} \wedge \mathbb{U} \wedge \mathbb{A'} \wedge \mathbb{A'}  + 
\Delta_4^a \cdot \mathbb{T} \wedge \mathbb{U} \wedge \mathbb{A} \wedge (\frac{\Delta'_4}{2\Delta_4}\mathbb{A'} +\dots) \\
&=(a\Delta_4^{a-1}\Delta'_4+ \Delta_4^a \cdot \frac{\Delta'_4}{2\Delta_4})\cdot \mathbb{T} \wedge \mathbb{U} \wedge \mathbb{A} \wedge \mathbb{A'}.
\end{split}
\end{align}
Thus, to get a parallel object, we get the condition 
$$
a\Delta_4^{a-1}\Delta'_4+ \Delta_4^a \cdot \frac{\Delta'_4}{2\Delta_4}=0
$$
Then $\Delta_4 \neq 0$ implies $\Delta_4^{a-1} \neq 0$  and we get $a=-\frac12$ or $\Delta'_4=0$. Let us remind the readers that $\Delta_4<0$, so the choice of the sign in the following statement makes sense.
\begin{prop}  \label{prop_general-tractor}
The curve $\gamma(t)$ satisfies $\Delta_4 \neq 0$ and $\Delta_5=0$ if and only if
 \begin{gather}
 \label{general-tractor}
 (-\Delta_4)^{-\frac12} \cdot \mathbb{T} \wedge \mathbb{U} \wedge \mathbb{A} \wedge \mathbb{A'}
 \end{gather}
  is parallel to the tractor connection along the curve. Moreover, if $\Delta_4$ is constant, then already the expression 
  \begin{gather} 
 \label{general-tractor2}
 T_4:= \mathbb{T} \wedge \mathbb{U} \wedge \mathbb{A} \wedge \mathbb{A'}
 \end{gather}
 is parallel to the tractor connection.
  \end{prop}
\begin{proof}
The above observations show that the conditions $\Delta_4 \neq 0$ and $\Delta_5=0$ are equivalent to the linear dependency of the tractors $\mathbb{T},  \mathbb{U},\mathbb{A},\mathbb{A'}$ and $\mathbb{A''}$, and the explicit combination takes the form \eqref{tract-comb-5}. 
The computation \eqref{derivative-4-tractor}  then gives that \eqref{general-tractor} is parallel with respect to the tractor connection. Conversely, assuming \eqref{general-tractor} is parallel, we compute
$$\nabla^{\wedge^4 \mathcal{T}}((-\Delta_4)^{-\frac12} \cdot \mathbb{T} \wedge \mathbb{U} \wedge \mathbb{A} \wedge \mathbb{A'})=(-\Delta_4^{-\frac12})\cdot \mathbb{T} \wedge \mathbb{U} \wedge \mathbb{A} \wedge (-\frac{\Delta'_4}{2\Delta_4} \mathbb{A}'+\mathbb{A}'')=0$$
and this implies that $\mathbb{A}'$ and $\mathbb{A}''$ are dependent modulo $\langle \mathbb{T},\mathbb{U},\mathbb{A} \rangle$. If $\Delta_4$ is a non-zero constant, then $\Delta'_4=0$ and $T_4$ is parallel from \eqref{derivative-4-tractor}.
\end{proof}  
Since $\Delta_4$ is a relative invariant,
the condition that $\Delta_4$ is a non-zero constant gives a special class of (parametrized) curves. In particular, conformal arc-length parametrization leads to $\Delta_4=-1$, \cite[Section 2.4.]{SZ19}.

\subsection{Conserved quantities from tractors} \label{sectract}
In \cite{GST21}, authors give a general method of how to use conformal Killing-Yano $2$-forms and the parallel $3$-tractor $T_3=\mathbb{T} \wedge \mathbb{U} \wedge \mathbb{A}$ to find conserved quantities of conformal circles on arbitrary conformal manifolds. The situation on the conformal sphere $S^n$ is simpler; let us outline here the situation. For interest, we give the resulting quantities on $S^n$ in Appendix \ref{A1}.

On the conformal sphere, each conformal Killing-Yano $2$-form corresponds (via BGG splitting) to an element of the tractor bundle for the representation $\rho_3:\frak{so}(n+1,1) \to \frak{gl}(\wedge^3\R^{n+2})$. 
It follows from the tractor theory that each such tractor is parallel with respect to the tractor connection. Indeed, in the case of $S^n$, the whole tractor space corresponds to solutions of the corresponding BGG operators, and they are all normal, i.e. coming from sections parallel to the tractor connection, \cite{CGH12,GZ}.
Now one can use these parallel tractors to construct conserved quantities. The tractor metric is parallel w.r.t. the tractor connection. The $3$-tractor $T_3$ is parallel w.r.t. the tractor connection along the conformal circle $\gamma$. Then pairing each such parallel tractor with the $3$-tractor $T_3$ w.r.t. to the tractor metric gives a conserved quantity along the conformal circle $\gamma$. 

On the conformal sphere, this directly generalises for curves satisfying $\Delta_4 \neq 0, \Delta_5=0$.
Analogously to the above case of cKY $2$-tensors, each conformal Killing-Yano $3$-tensor corresponds to an element of the tractor space $\wedge^4\R^{n+2}$,  and these are parallel w.r.t. the tractor connection.  We showed in Proposition  \ref{prop_general-tractor} that $4$-tractor $ (-\Delta_4)^{-\frac12} \cdot \mathbb{T} \wedge \mathbb{U} \wedge \mathbb{A} \wedge \mathbb{A'}$ from $\eqref{general-tractor}$ is parallel along the curve $\gamma$ with $\Delta_4 \neq 0, \Delta_5=0.$ Then the pairing of the parallel tractors and the $4$-tractor w.r.t. the tractor metric gives a conserved quantity along  $\gamma$. 
If, moreover, $\Delta_4$ is constant, then pairing with $T_4$ already gives a conserved quantity. 

Let us describe these quantities explicitly here.
 We fix the basis 
$e_{abcd} := e_a \wedge e_b \wedge e_c  \wedge e_d$
with $0 \leq a<b<c<d \leq N$ of $\wedge^4\R^{n+2}$
for the standard basis $e_a$, $a=0,\dots, N$  of $\R^{n+2}$.

Firstly, we find the explicit form of the $4$-tractor $T_4 = \mathbb{T} \wedge \mathbb{U} \wedge \mathbb{A} \wedge \mathbb{A'}.$
 
 \begin{prop} \label{4tractprop}
 The $4$-tractor $T_4 = \mathbb{T} \wedge \mathbb{U} \wedge \mathbb{A} \wedge \mathbb{A'}$ takes the explicit form
 \begin{gather} \label{4tractor}
T_4=u^{-4} U^{[i}A^j(A')^{k]}
e_{0ijk}
-3 u^{-4}\langle U,A \rangle U^{[i}A^{j]} \: 
e_{0ijN}+
u^{-2} U^{[i} (A')^{j]} \: 
e_{0ijN},
 \end{gather}
where we use Einstein summation convention,  $[ab..c]$ denotes the anti-symmetrization over indices in the bracket, $\langle\ , \ \rangle$ denotes the standard scalar product on $\R^n$ and $i,j,k=1, \dots, n$.
 \end{prop}

\begin{proof} 
One can see from the explicit forms of tractors 
 \eqref{trT}, \eqref{trU}-\eqref{trAbar} that $\mathbb{T}$ can only add the term with $e_0$, so no other tractor adds a term with $e_0$. Then $\mathbb{U}$ can only add a term with $e_i$, thus the expression starts  with $u^{-1}e_0 \wedge u^{-1} U^ie_i$.
Then possible additional terms coming from $\mathbb{A}$ and $\mathbb{A'}$ come with $e_i$ or $e_N$. Moreover, it holds $U \wedge U =0$ and we always get $U$ from $\mathbb{U}$, so we can only consider terms from $\mathbb{A}$ and $\mathbb{A'}$ with $A$ and $A'$.
Altogether, we have the following 
expression
\begin{gather*}
u^{-1}e_0 \wedge u^{-1} U^ie_i \wedge \big( u^{-1}A^je_j-u e_{N}  \big) 
\wedge  \big( -3 u^{-3}\langle U,A \rangle A^ke_k+ u^{-1}(A')^ke_k\big)
\end{gather*}
where we use the Einstein summation notation and $i=1,\dots,n$.
Since $A \wedge A=0$, some more terms vanish and we get 
\begin{align*}
T_4
&=
u^{-1}e_0 \wedge u^{-1} U^ie_i \wedge  u^{-1}A^je_j  
\wedge  u^{-1}(A')^ke_k \\ & \ \ \ +
u^{-1}e_0 \wedge  u^{-1} U^ie_i \wedge (-u) e_{N} 
\wedge  (-3 u^{-3})\langle U,A \rangle A^ke_k \\ & \ \ \ +
u^{-1}e_0 \wedge u^{-1} U^ie_i \wedge (-u) e_{N}  
\wedge  u^{-1}(A')^ke_k \\ &=
u^{-4}e_0 \wedge  U^ie_i \wedge  A^je_j  
\wedge  (A')^ke_k \\ & \ \ \ +
u^{-4}e_0 \wedge   U^ie_i \wedge  e_{N} 
\wedge 3 \langle U,A \rangle A^ke_k \\ & \ \ \  -
u^{-2}e_0 \wedge   U^ie_i \wedge  e_{N}  
\wedge  (A')^ke_k.
\end{align*}
The order of wedges of standard basis elements does not respect the lexicographic ordering, so we have to reorder the terms to get a description in the lexicographic basis $e_{abcd}.$ This leads to the reordering of coefficients, and summarising them together leads exactly to the anti-symmetrisation, so we get the formula for $T_4$ in the statement. 
\end{proof}

Let us now find the conserved quantities. 
In Proposition \ref{prop_conservedqtractor} below, we use the elementary alternating tensor field to express the quantities.  Recall that the elementary alternating tensor field  of rank $m,$ denoted by $\ep^{i_1 \dots i_m},$ is defined by its action on $m$ vector fields, $Y, \dots, Z,$ as
\[ \ep^{i_1 \dots i_m} (Y, \dots, Z) 
= 
\det \left(
\begin{array}{ccc}
Y^{i_1} & \dots &  Z^{i_1}  \\
\vdots & { } &  \vdots \\
Y^{i_m}& \dots &  Z^{i_m}
\end{array}
\right),
\]
where $Y^j$ and $Z^j$ are components of $Y$ and $Z,$ respectively.

\begin{prop} \label{prop_conservedqtractor}
Consider the basis $e_{abcd}$ for $0 \leq a<b<c<d \leq N$ of $\wedge^4\R^{n+2}$. Pairing the $4$-tractor $T_4$ from  \eqref{4tractor} and the parallel tractor corresponding to $e_{abcd}$ with respect to the tractor metric, which is given by $G$ from \eqref{met}, determines the conserved quantity $Q^{abcd}$ of curves satisfying $\Delta_5=0$ and $\Delta_4$ non-zero constant as follows: 
\begin{align}
 Q^{0ijN} &= 3u^{-4}\langle U,A\rangle \epsilon^{ij}(U,A)-u^{-2} \epsilon^{ij}(U,A') 
 +u^{-4} \epsilon^{ijl}(U,A,A') X_l \label{Q0ijN}
\\
\begin{split}
Q^{0ijk} &= 3u^{-4}\langle U,A \rangle \epsilon^{ijk}(X,U,A)-u^{-2} \epsilon^{ijk}(X,U,A')   \\
&+\frac{1}{2} |X|^2 u^{-4} \epsilon^{ijk}(U,A,A') +
 u^{-4} \ep^{ijkl}(X,U,A,A') X_l
\label{Q0ijk} 
\end{split}
\\
 Q^{ijkN}& = u^{-4}\epsilon^{ijk}(U,A,A')
\label{QijkN} \\
Q^{ijkl } &= u^{-4}\epsilon^{ijkl}(X,U,A,A'). \label{Qijkp}
\end{align}
The quantities for an arbitrary curve satisfying $\Delta
_5=0, \Delta_4 \neq 0$ are given by the same formulae multiplied by $(-\Delta_4)^{-\frac12}$, see \eqref{Delta} for $\Delta_4$.
\end{prop}
\begin{proof}
The standard action \eqref{action2} differs for $e_0$, for $e_i$, $i=1,\dots,n$ and for $e_N$. 
This leads us to the decomposition of the basis of $\wedge^4 \R^{n+2}$ as follows. We get
\begin{enumerate}
\item
 $\frac12 n(n-1)$ elements of the form $e_{0ijN}$, 
\item 
$\frac16 n(n-1)(n-2)$ elements of the form   $e_{0ijk}$, 
\item
 $\frac16 n(n-1)(n-2)$ elements of the form  
$e_{ijkN}$, and
\item
 $\frac{1}{24}n(n-1)(n-2)(n-3)$ elements of the form  
$e_{ijkl}$,
\end{enumerate}
where $i,j,k,l=1,\dots, n$.
For each basis element, i.e. for each coordinate, we compute the expression $\exp(-\rho_4(J))$ for the action $\rho_4: \frak{so}(n+1,1) \to \frak{gl}(\wedge^{4}\R^{n+2})$ given by tensoriality of the standard action, see  \eqref{exp} for $J$.
 Thus we compute
 \begin{gather} \label{expw}
 w-\rho_4(J)(w)+\frac12  \rho_4^2(J)(w)
 \end{gather}
 for coordinates $w$,
 and we always choose one of the coordinates equal to $1$ and the rest zero.  
Direct computations using the formula \eqref{action2} and tensoriality give the following actions $\rho_4(J)$, where sums are generically over $l =1, \dots, n$ and we emphasize additional restrictions below the sums
\begin{align*}
e_{0ijN}& \mapsto \sum_{l <i} X^l e_{l ijN} - 
 \sum_{i<l<j} X^l e_{il jN} + \sum_{l >j} X^l  e_{ijl N}
\\
e_{0ijk}&\mapsto \sum_{l <i} X^l  e_{l ijk} - \sum_{i<l <j} X^l  e_{il jk} + 
 \sum_{j<l <k} X^l  e_{ijl k}  
- \sum_{k <l} X^l  e_{ijkl } 
\\ &
\ \ \ \ -  X^i e_{0jkN}+ 
 X^j e_{0ikN} -X^k e_{0ijN}
\\
 e_{ijkN}&\mapsto 0
 \\
e_{ijkl }&\mapsto X^i e_{jkl N} -X^j e_{ikl N} + 
X^k  e_{ijl N} - X^l  e_{ijkN}.
 \end{align*}
 Having these, we compute the second powers, aka compositions $\rho_4^2(J)$ as
 \begin{align*}
 e_{0ijN}&\mapsto 0\\
e_{0ijk}& \mapsto \sum_{l =1}^n\Big( (X^l )^2-2\big((X^i)^2+(X^j)^2+(X^k)^2\big) \Big)e_{ijkN} \\
& \ \ \ \  - 2 X^i\big(\sum_{l <j,\: l \neq i}X^l e_{l jkN}-\sum_{j<l <k}X^l e_{jl kN}+\sum_{k<l }X^l e_{jkl N}   \big) \\ 
&\ \ \ \  +2X^j\big(\sum_{l<i} X^l  e_{l ikN}-\sum_{i<l <k,\: l  \neq j}X^l  e_{il kN}+\sum_{l >k} X^l  e_{ikl N}   \big)\\ 
&\ \ \ \ -2X^k\big(\sum_{l <i}X^l  e_{l ijN}-\sum_{i<l <j}X^l  e_{il jN}+\sum_{l >j,\: l \neq k} x_l  e_{ijl N}\big)
\\ 
e_{ijkN}&\mapsto 0
\\
e_{ijkl }&\mapsto 0.
  \end{align*}
Then it is a direct computation to sort these terms and put them together according to \eqref{expw} for each basis element $e_{abcd}$. The pairing of the resulting expressions and the $4$-tractor $T_4$ from \eqref{4tractor} with respect to the tractor metric $G$ then means \emph{swapping $0$ and $N$ in one of the two expressions and standard pairing.} Doing these, we get exactly the formulae \eqref{Q0ijN} - \eqref{Qijkp}  from the statement. Let us emphasise that in the case $e_{ijkN}$, the pairing leads to the quantity $Q^{ijkN}$ but with the minus sign, so we ignore the sign and simply consider $Q^{ijkN}$.

It then follows that the same computation works for pairing of $(-\Delta_4)^{-\frac12} \cdot T_4$ with the parallel tractor basis, and leads to the same expressions just multiplied by $(-\Delta_4)^{-\frac12}$. These are then constant for curves satisfying $\Delta_5=0$, $\Delta_4 \neq 0$. In particular, if $\Delta_4$ is non-zero constant, this multiple is a constant we can omit this multiple in quantities. 
\end{proof}

\begin{exam}{\it{Quantities in dimension $3$.}}
We have $5$ basis elements $$e_0=e_{0123},  e_1=e_{0124},  e_2=e_{0134}, e_3=e_{0234},  e_4=e_{1234}$$ and denote corresponding coordinates $w^0, \dots, w^4$. 
The standard action \eqref{action2} and tensoriality then give
\begin{gather*}
e_0 \mapsto  -X^1e_3+X^2e_2-X^3e_1,\ \ \  e_1 \mapsto X^3e_4, \ \ \   e_2 \mapsto -X^2e_4, \ \ \   e_3 \mapsto X^1  e_4, \ \ \   e_4 \mapsto 0.
\end{gather*}
In the coordinates $(w^a)$, $a=0,1,2,3,4$, we can write the corresponding transformation matrix $\rho_4(J)$ and thus $\exp(-\rho_4(J))$ as
$$
\rho_4(J)=\left( \begin{smallmatrix} 
0&0&0&0&0 
\\ -X^3&0&0&0&0 
\\X^2&0&0&0&0 
\\-X^1&0&0&0&0 
\\0&X^3&-X^2&X^1&0 
\end{smallmatrix} \right),
\ \
\exp(-\rho_4(J))=\left( \begin{smallmatrix} 
1&0&0&0&0 
\\ X^3&1&0&0&0 
\\-X^2&0&1&0&0 
\\X^1&0&0&1&0 
\\-\frac12\big( (X^1)^2+(X^2)^2+(X^3)^2\big)&-X^3&X^2&-X^1&1 
\end{smallmatrix} \right).
$$
Thus, we describe elements of $\wedge^4\R^5$ in exponential coordinates  
$\exp(-\rho_4(J))(w)$ as
\begin{gather*}
[\  w^0,\ \  w^0X^3+w^1,\ \  -w^0X^2+w^2,\ \  w^0X^1+w^3, 
\\
-\frac12 \big( (X^1)^{2}+(X^2)^{2}+(X^3)^{2} \big) 
{ w^0}-{ w^3}{ X^1}+{ w^2}{ X^2}-{ w^1}{ X^3}+{ w^4} \  ]
\end{gather*}
Now, pairing these with $T_4$ with respect to the tractor metric gives conserved quantities. We get a generic family of quantities so that for each $a=0, \dots, 4$, we choose $w^a=1$ and the remaining coordinates zero.
 
We give the resulting quantities as follows.  Below, the wedge product of three vectors is the scale given by the determinant of the corresponding matrix, and the wedge product of two vectors is their cross product. 
Then we get three quantities of type \eqref{Q0ijN}
written as the vector $Q=(Q^{0124},-Q^{0134},Q^{0234})$ as follows, where $X_{swap}=(X^3,X^2,X^1)$ 
\begin{gather*}
Q=u^{-4} \, U \wedge A \wedge A' \, X_{swap}-u^{-2} \, U \wedge A'+3u^{-4}\: \langle U,A \rangle \, U \wedge A.
\end{gather*} 
We have one quantity of type \eqref{Q0ijk}
\begin{gather*}
Q^{0123}=3u^{-4}\langle U,A \rangle \: X \wedge U \wedge A-u^{-2} X \wedge U \wedge A'+\frac12 | X |^2u^{-4}\: U \wedge A \wedge A'.
\end{gather*}
Let us note that in this dimension, the expression simplifies as the wedge product of four vectors vanishes. There is one quantity of type \eqref{QijkN}
\begin{gather*}
Q^{1234}=u^{-4}\: U \wedge A \wedge A'
\end{gather*}
and quantities of type \eqref{Qijkp} do not appear.
\end{exam}

\begin{exam} \label{exam_logspiral} {\it{Logarithmic spirals.}}
Let us consider the family of logarithmic spirals as follows
\begin{gather}
\label{logspiral} 
\wh X(t) = e^t  \cos(ct) P_0 +  e^t  \sin(ct) Q_0 + R_0,
\end{gather}
where $c$ is a constant, $P_0, Q_0, R_0$ are constant vectors, $\langle P_0, Q_0 \rangle =0$ and ${|P_0| = |Q_0|}$, \cite{DK21}.
We compute quantities \eqref{Q0ijN}-\eqref{Qijkp} for these curves.
Direct computation gives 
\begin{align*}
U(t)&=e^t(\cos{(ct)}-\sin{(ct)})P_0+e^t(\cos{(ct)}+\sin{(ct)})Q_0,\\
A(t)&=e^t((1-c^2)\cos{(ct)}-2c\sin{ct})P_0+e^t((1-c^2)\sin{(ct)}+2c\cos{(ct)}Q_0,\\
A'(t)&=e^t((c^3-3c)\sin{(ct)}+(-3c^2+1)\cos{(ct)})P_0\\
&-e^t((c^3-3c)\cos{(ct)}+(3c^2-1)\sin{(ct)})Q_0.
\end{align*}
Then we compute
\begin{gather*}
   u^2=e^{2t}(c^2+1)|P_0|^2, \ \ \ \ \langle U,A \rangle = e^{2t}(c^2+1)|P_0|^2, \ \ \ \ \langle U,A' \rangle = e^{2t}(c^4-1)|P_0|^2
\end{gather*}
thanks to the facts $|P_0|=|Q_0|$ and $\langle P_0, Q_0 \rangle=0$.
Moreover, the expression $\epsilon^{ijk}(P_0,Q_0,R_0)$ (up to the order of vectors) is the only non-vanishing sub-determinant of third order; other possibilities contain at least two same vectors and vanish. Analogously, expressions of the fourth order always vanish. Altogether, we compute
\begin{gather*}
Q^{0ijN}=\frac{c}{|P_0|^2}\epsilon^{ij}(P_0,Q_0), \ \ \ \ Q^{0ijk}= \frac{c}{|P_0|^2}\epsilon^{ijk}(P_0,Q_0,R_0), \ \ \ \ Q^{ijkN}=Q^{ijkl}=0.
\end{gather*}
In particular, the quantities are constant along logarithmic spirals \eqref{logspiral}.

Moreover, we can find explicitly that along the curves $\wh X(t)$, the tractor $\mathbb{A}$ takes form
\begin{gather*}
\mathbb{A}= \frac{e^{-t}}{|P_0|\sqrt{c^2+1}}e_0-\frac{\sqrt{c^2+1}(\cos{(ct)}P_0^i+\sin{(ct)}Q_0^i)}{|P_0|}e_i -e^{t}|P_0|\sqrt{c^2+1}e_N.
\end{gather*}
Then the square of the length of this tractor for the tractor metric is $\alpha_1=c^2-1$. Similarly, one computes the tractor $\mathbb{A'}$ and the square of its length $\alpha_2=c^4-c^2+1$. In particular, $$\Delta_4=\alpha_1^2-\alpha_2=-c^2 \ \  {\rm and \  thus} \ \ (-\Delta_4)^{-\frac12} =\frac{1}{c}.$$ 
It is a more nasty but analogous computation to find that $\Delta_5=0$.
\end{exam}

Let us note that one of the possible definitions of logarithmic spirals is via invariants, \cite[Section 4]{SZ19}. We say that a curve is a spiral if $\kappa_1$ is constant and $\Delta_5=0$, where 
\begin{gather*}
 \kappa_1=-\frac12 (-\Delta_4)^{-\frac52}(\alpha_1 \Delta_4^2-\frac12\Delta_4\Delta''_4+\frac{9}{16}(\Delta'_4)^2)
\end{gather*}
is the first conformal curvature. Note this is an absolute invariant, i.e. independent of reparametrisation. See \cite{Sulanke,SZ19} for more details on conformal curvatures.

If we assume $\Delta_4$ constant, we also get $\alpha_1$ constant from the assumption that $\kappa_1$ is constant. The condition $\alpha_1$ constant gives a special parametrization which generalizes projective parametrization;  the projective parametrization corresponds to $\alpha_1=0$. 
In general, if $\Delta_4$ and $\alpha_1$ are constant, then $\kappa_1$ is constant, but the opposite direction does not necessarily hold in general. A different viewpoint on logarithmic spirals was introduced in \cite{E23}, where the author focuses on homogeneous curves and studies their symmetry algebras.

\section{Conserved quantities for the conformal Mercator equation} \label{sec:firstintegralF}
Here we come to our second generalization of conformal circles on the conformal sphere, that is, the class of parametrized curves that are solutions of the flat conformal Mercator equation.   
  Note that the logarithmic spiral \eqref{logspiral} is also a solution of the conformal Mercator equation \cite{DK21}.  Thus the generalization here and the one in Section \ref{sec:curvetractor} share some common curves.


\subsection{Conformal Mercator equation in the conformally-flat case} \label{sec:confMeq}

Choosing the Euclidean metric $\de$ in the conformal class, the conformal Mercator equation \cite{DK21} is given by 
 \be \label{confMeqflat} \frac{dC}{dt} = 0, \; C = u^{-2} \left( A' - u^{-2}|A|^2 U -2 u^{-2} \langle A,U\rangle A + 4 u^{-4} \langle A,U \rangle^2 U - 2 u^{-2} \langle A',U \rangle U \right),\ee
 where  $U$ and $A$ are respectively the velocity and acceleration vector fields of a curve $\g(t),$ and 
  $A' = \frac{d A}{dt},$ as before.  

Equation \eqref{confMeqflat} is the Euler-Lagrange equation of the Lagrangian \cite{DK21}
\be \label{Lagflat}  L =  \frac{d}{dt} \left( u^{-2}\langle U,A \rangle \right) + \frac{1}{2} u^{-2}|A|^2 - u^{-4} \langle U,A \rangle^2,\ee
where the variation and its derivatives up to second order are assumed to vanish at end points.  This fact enables us to apply Noether's principle to find conserved quantities (or first integrals - functions that are constant on all solutions) of \eqref{confMeqflat}.


\subsection{Conserved quantities from conformal Killing vector fields}

Our derivation of conserved quantities
is based on a simplest version of Noether's theorem, which we shall describe below.  (See e.g. \cite{A89, O93}.)  Suppose an ODE is the Euler-Lagrange equation of a functional
\[ I(\g) = \int_{t_0}^{t_1} L(\g(t)) \; dt,\]
meaning that all solutions of the ODE are critical curves of $I(\g).$  
A one-parameter group of symmetries of the Lagrangian $L$ is a one-parameter group of transformations  $\{\phi_s\}$  
that preserves $L$ in the sense that 
\be \label{Lagsym} L(\g(t)) = L( \phi_s(\g(t))),\ee
where $L(\g(t))$ denotes the Lagrangian evaluated on a curve $\g,$ while $L( \phi_s(\g(t)))$ denotes the one evaluated on the transformed curve $\phi_s(\g(t)).$
Noether's theorem states that there exists an associated conserved quantity for solutions of the Euler-Lagrange equation for each such symmetry group.  
  Note that if  $\{\phi_s\}$ is a symmetry group of the Lagrangian, then it is also a symmetry group of the Euler-Lagrange equation, which means that if $\g(t)$ is a solution of the ODE, then $\phi_s(\g(t))$ is also a solution.  

Based on Noether's principle, let us now derive an expression for the conserved quantities of solutions of the conformal Mercator equation \eqref{confMeqflat}.   Let $\g$ be a solution of \eqref{confMeqflat} and $\{\phi_s\}$ a one-parameter group of symmetries of the Lagrangian \eqref{Lagflat}.  Consider a one-parameter family of solution curves $\G: [s_0, s_1] \times [t_0, t_1] \to \R^n$ given by 
$\G(s,t)=\phi_s(\g(t)).$  Choose the parameter $s$ such that $0 \in [s_0, s_1]$ and  $\phi_0 = {\rm Id},$ i.e.  $\G(0,t) = \g(t).$  If we restrict the range of the parameter $t$ so that the generator of $\{\phi_s\}$ is nowhere tangent to $\g,$ then for a sufficiently small range of $s$ around $0,$ the family $\G$ is a $2$-dimensional submanifold of $\R^n.$

Now, let $U$ (where we abuse the notation) denote the vector field on $\G$ whose restriction to each solution curve $\phi_s(\g(t))$ ($s$ fixed) is the velocity of the curve.  Also, let $V$ be the generator of $\{\phi_s\}$ restricted to $\G.$  By definition, $U$ is invariant under the flow of $V,$ and thus $[U,V] = 0.$

Next, consider the Lagrangian function on $\G$ given by
\[ L(s,t) : = L( \phi_s(\g(t))) =  \nb_U \left( u^{-2} \langle U,A \rangle \right) + \frac{1}{2} u^{-2}|A|^2 - u^{-4} \langle U,A \rangle^2,\]
where $\nb$ is the Levi-Civita connection of $\de$ and $\nb_U$ the covariant derivative along each curve $\phi_s(\g(t)).$  
 Equation \eqref{Lagsym} implies that $V(L(s,t)) = 0.$ Then, by direct calculation with some rearrangements, it can be shown that 
\be \label{VL} 0 = V(L) = \nabla_U F + \langle \nabla_U C, V \rangle, \ee
where $F$ is a function on $\G$ given by
\[F = \nb_U \langle W, \nb_U V \rangle + \langle  \nb_U W, \nb_U V \rangle - \langle C,V \rangle, \quad \mbox{with} \quad W = u^{-2} U, \]
and the expression for $C$ is given by \eqref{confMeqflat}.
Let us note that in obtaining \eqref{VL}, we have used the flatness of $\de$ and the fact that $[U,V] = 0.$  These imply that $\nabla_V U = \nabla_U V$ and $\nabla_V \nabla_U U= \nabla_U \nabla_V U.$ 

Now restricting \eqref{VL} to each solution curve $\phi_s(\g(t))$ which satisfies \eqref{confMeqflat}, $\nabla_U C =0,$ we have  that $\nabla_U F = 0$ upon restriction.  Hence $F$ is constant on each solution.

The proposition below follows from the fact that 
 the Lagrangian \eqref{Lagflat} is invariant under one-parameter groups of conformal transformations.\footnote{One can verify directly that the Lagrangian \eqref{Lagflat} satisfies \eqref{Lagsym} for any one-parameter group of conformal transformations.}

\begin{prop} \label{prop_firstintegralF}
Let $\g: t \mapsto \g(t)$ be a smooth curve, and $U$ and $A$ be its velocity and acceleration vector fields, respectively. Let $V$ be a conformal Killing vector field 
restricted to $\g.$   Then the function 
\be \label{F} F = \frac{d}{dt} \langle W, V' \rangle + \langle W', V' \rangle  - \langle C,V \rangle,  \qquad W = u^{-2} U,  \ee
where $C$ is given in {\rm \eqref{confMeqflat}} and $\,{}'$ denotes $\frac{d}{dt},$ is constant on any curve $\g$ that satisfies the conformal Mercator equation {\rm \eqref{confMeqflat}.}
\end{prop}

\begin{proof}  The proposition follows from the derivation based on Noether's principle as discussed above.  However, it can also be proved by directly calculating $\nb_U F = F'$ and show that it vanishes if the conformal Mercator equation and the conformal Killing equation hold.  Let us outline the calculation here briefly.

First recall that a conformal Killing vector field $V$ satisfies the  conformal Killing equation 
\be \label{confKeq} \nb_{(i} V_{j)} = \ll \, \de_{ij},  \qquad \ll = \frac{1}{n} \nb_i V^i,\ee
where we use the Einstein summation convention.   
This implies that upon restriction to a curve $\g,$ the vector field $V$ satisfies
\be \label{confKvfW}\langle \nb_U V, U \rangle = \ll \, u^2, \quad \mbox{i.e.} \quad \langle \nb_U V, W \rangle = \ll.\ee
Now rewriting $F$ as
\[F = \nb_U \langle W, \nb_U V \rangle -2 \langle W, A \rangle \langle W, \nb_U V \rangle + \langle u^{-2} A , \nb_U V \rangle- \langle C,V \rangle\]
gives
\begin{eqnarray*}\nb_U F  &=& \nb_U^2 \langle W, \nb_U V \rangle -2 \langle W, A \rangle \nb_U \langle W, \nb_U V \rangle  - u^{-2} |A|^2 \langle W, \nb_U V \rangle \\ 
&{}& + \langle u^{-2} A, \nb_U^2 V \rangle -  \langle \nb_U  C,V \rangle,
\end{eqnarray*}
where the last term vanishes by the conformal Mercator equation $\nb_U C=0.$

Using the conformal Killing equation \eqref{confKvfW}, we have that 
\be \label{nbF} \nb_U F  = \nb_U^2 \ll -2 \langle W, A \rangle \nb_U \ll  - u^{-2} |A|^2 \ll + \langle u^{-2} A, \nb_U^2 V \rangle. \ee
Then the last term of \eqref{nbF} can be computed using a differential consequence of the conformal Killing equation \eqref{confKeq} and the fact that $\nb$ is flat.  This results in 
\[\nb_U F  = \nb_U^2 \ll -   d \ll (A) = U^i U^j \nb_i \nb_j \ll.\]
Now, if $V$ is Killing, then $\ll$ is zero.  Also, we have that $\ll = \frac{1}{n} \nb_i V^i$ is constant for the generator of dilatation, and is linear in spacetime coordinates for those of special conformal transformations.  It follows that $\nb_U F =0$ for all conformal Killing vector fields.  
\end{proof}


\subsection{Expressions for conserved quantities and examples}

The conformal Killing vector fields 
can be divided into four types, namely the generators of translation, rotation, dilatation and special conformal transformation. Using the same notation as in Proposition \ref{prop_firstintegralF}, below we give an expression for the conserved quantity  \eqref{F}
for a conformal Killing vector field $V$ of each type.  Note that tautologically $C$ is a constant vector field along any solution curve of \eqref{confMeqflat}.
Also, the first term of \eqref{F} vanishes for all conformal Killing vector fields $V$ except for the generator of a special conformal transformation.  

\begin{prop} \label{prop_FTFRFDFS}
Let $(x^i)$ be the standard coordinates for $\R^n,$ and $\g$ be a solution to the conformal Mercator equation {\rm \eqref{confMeqflat}} given by $(x^i) = (\g^i(t)) =: (X^i(t)).$

Let $V_T,$ $V_R,$ $V_D$ and $V_S$ be the generators of translation, rotation, dilatation and special conformal transformation, respectively, given by
\[V_T  = T^i  \frac{\p}{\p x^i}, \quad V_R = {R_j}^i x^j  \frac{\p}{\p x^i}, \quad  V_D = a \, x^i  \frac{\p}{\p x^i}, \quad V_S =  \left(x_j x^j S^i - 2 S_j x^j x^i \right)  \frac{\p}{\p x^i}, \]
where $T^i, {R_j}^i, a, S^i$ are constant and arbitrary, with $R_{ij} = - R_{ji},$ and the indices are raised and lowered by $\de.$  Also, let $T:=T^i  \frac{\p}{\p x^i}$ and $S:=S^i  \frac{\p}{\p x^i}.$

Then the corresponding conserved quantities, $F$ from {\rm \eqref{F}}, for $V_T,$ $V_R,$ $V_D,$ $V_S$ are respectively  given by
\begin{eqnarray}
F_T  &=& - \langle C,T \rangle  \label{FT} \\
F_R &=& R_{ij} \left(u^{-2} U^i A^j + C^i X^j \right)\label{FR} \\
F_D &=& -a  \left(u^{-2} \langle U,A \rangle + \langle C,X \rangle\right)\label{FD} \\
F_S &=&  2 \langle S, {\mc Y} \rangle \quad \mbox{with} \label{FS} \\
{\mc Y} &=& u^{-2} \langle U, X \rangle A - \left(1+ u^{-2} \langle A, X \rangle \right) U + \left(u^{-2} \langle U,A \rangle + \langle C,X \rangle \right) X - \frac{|X|^2}{2} C, \nonumber
\end{eqnarray}
where we use $X$ to denote the vector field $X^i  \frac{\p}{\p x^i}$ along $\g.$
\end{prop}

\begin{proof}
The proposition follows by direct substitution of $V_T,$ $V_R,$ $V_D,$ $V_S$ into \eqref{F}, where upon restriction to $\g$ we have $V_R = {R_j}^i X^j  \frac{\p}{\p x^i},$ $V_D = a \, X$ and $V_S =  |X|^2 S - 2 \langle S,X \rangle X.$
\end{proof}

Note that for translations, choosing $T$ to be each coordinate vector field $T = \frac{\p}{\p x^i},$  the quantity $F_T = -C^i$ is just minus each component of the tautologically constant vector field $C.$ 

\begin{exam}  
It was shown in \cite{DK21} that the logarithmic spiral \eqref{logspiral},
 \[ \wh X(t) = e^t  \cos(ct) P_0 +  e^t  \sin(ct) Q_0 + R_0,\]
is the general solution of the reduction of \eqref{confMeqflat} under the condition $C=0.$  For this solution, the conserved quantities \eqref{FT} - \eqref{FS} are given by
\[F_T =0, \quad F_R = \frac{c}{|P_0|^2}R_{kl}P_0^kQ_0^l, \quad F_D = -a, \quad F_S = 2 \langle S, -c \frac{\langle Q_0, R_0 \rangle}{|P_0|^2}P_0 + c \frac{\langle P_0, R_0 \rangle}{|P_0|^2}Q_0+ R_0 \rangle.\]

Let us note the similarity between the quantity $Q^{0ijN}=\frac{c}{|P_0|^2}\epsilon^{ij}(P_0,Q_0)$ of Example \ref{exam_logspiral} and $F_R$ here.  In fact, for a fixed $i <j,$ choosing the anti-symmetric matrix $R$ to be such that $R_{ij} =1,$ $R_{ji}=-1$ and all other entries vanish, one has $F_R = Q^{0ijN}.$  We will come back to this example again in Section \ref{sec:Hamiltonian} after we establish relations between the conserved quantities of Proposition \ref{prop_conservedqtractor} and Proposition \ref{prop_FTFRFDFS}.
\end{exam}

\begin{exam}
Recall that under the stereographic projection, the inverse image of logarithmic spirals on a plane are loxodromes on $S^2,$ which are curves that cut all meridians at a constant angle.   Let us consider the logarithmic spiral \eqref{logspiral} in two dimensions. 

Without loss of generality, one can choose $P_0 = (1,0),$ $Q_0 = (0,1)$ and $R_0 = (0,0).$  Then the inverse image of the curve is a loxodrome which cuts all meridians at the angle $\cos^{-1}(1/\sqrt{1+c^2}).$  The conserved quantity which corresponds to this constant angle is $F_R$;  choosing $R = \left( \begin{smallmatrix} 0&1 \\ -1 &0 \end{smallmatrix} \right),$ we have that $F_R = c.$ 
\end{exam}

\begin{exam}
Due to conformal invariance of the conformal Mercator equation \eqref{confMeqflat}, the special conformal transformation of \eqref{logspiral},
\be \label{523} X(t) = \frac{\wh X(t) - |\wh X(t)|^2B}{1-2\langle \wh X(t), B \rangle+|B|^2 |\wh X(t)|^2},\ee
where $B$ is a constant vector, is also a solution of \eqref{confMeqflat} \cite{DK21}.

Below we give the values of the conserved quantities \eqref{FT} - \eqref{FS}, evaluated on \eqref{523}:
\begin{eqnarray*}
F_T  &=& - \langle C,T \rangle, \qquad \mbox{where} \\
C &=& \frac{2}{|P_0|^2}\Big(c\langle Q_0, |B|^2 R_0 - B \rangle P_0 - c \langle P_0, |B|^2 R_0 - B \rangle Q_0 - |B|^2 |P_0|^2 R_0 \\
&{}& \qquad + \big(2 c \langle P_0, R_0 \rangle\langle Q_0, B \rangle - 2 c \langle Q_0, R_0 \rangle\langle P_0, B \rangle + (2\langle R_0, B \rangle-1)|P_0|^2 \big)\,B  \Big) \\
F_R &=& \frac{c}{|P_0|^2} [P_0, Q_0]_R - 2 [B, R_0]_R + \frac{2c}{|P_0|^2}\Big([Q_0, B]_R\langle P_0, R_0 \rangle - [P_0, B]_R\langle Q_0, R_0 \rangle \Big),\\
&{}& \mbox{where} \quad  [P_0, Q_0]_R  := R_{ij}P_0^iQ_0^j, \; \mbox{etc.} \\
F_D &=& - a \left( 1 - 2\langle R_0, B \rangle +\frac{2c}{|P_0|^2}\big(\langle P_0, B \rangle\langle Q_0, R_0 \rangle - \langle Q_0, B \rangle\langle P_0, R_0 \rangle\big) \right) \\
F_S &=&  2 \langle S, {\mc Y} \rangle, \quad \mbox{where} \quad
{\mc Y} = -c \frac{\langle Q_0, R_0 \rangle}{|P_0|^2}P_0 + c \frac{\langle P_0, R_0 \rangle}{|P_0|^2}Q_0+ R_0.
\end{eqnarray*}
\end{exam}
 
 Let us end this section by noting that, as conformal circles are also solutions of the flat conformal Mercator equation \eqref{confMeqflat}, the function $F$ \eqref{F} also gives a conserved quantity for conformal circles for each conformal Killing vector field $V.$  We summarize this result in Appendix \ref{app_circleV}.


\section{Hamiltonian formulation and relation between conserved quantities} \label{sec:Hamiltonian}

In this section we write our conserved quantities of Section \ref{par-trac} and Section \ref{sec:firstintegralF} in terms of the Hamiltonian phase space variables associated to the conformal Mercator equation. The advantage of this formulation is that the expressions for the quantities are all simplified to polynomials.  This helps us to notice and write down relations between them.  In particular, we manage to write down the conserved quantities of Proposition \ref{prop_conservedqtractor} in terms of 
those of Proposition \ref{prop_FTFRFDFS}.

\subsection{Hamiltonian formalism}

Here we briefly review the Hamiltonian formulation of the conformal Mercator equation \eqref{confMeqflat} 
given in \cite{DK21}.  Using the Ostrogradsky construction to the second order derivative Lagrangian \cite{BE90}
\[L_1 =  \frac{1}{2} u^{-2} |A|^2 - u^{-4}\langle U,A \rangle^2,\]
which is the Lagrangian $L$ \eqref{Lagflat} without the first term\footnote{Recall that the conformal Mercator equation \eqref{confMeqflat} is the Euler-Lagrange equation of both Lagrangians $L$ and $L_1,$ as the first term of $L$ \eqref{Lagflat} is integrated to a boundary term.  In Section \ref{sec:firstintegralF}, we use $L$ as we require conformal invariance.  Here we simply describe a Hamiltonian system which is equivalent to \eqref{confMeqflat}.}, the corresponding Hamiltonian phase space is of $4n$ dimensions, with $X^i$ and $U^i,$
$i = 1, \dots, n,$ as the first $2n$ canonical coordinates.   
The other $2n$ coordinates are given by the momenta $\mP = (\mP_i)$ and $\mR = (\mR_i)$ conjugate to $X$ and $U,$ respectively, which are defined by
\[\mP_i = \frac{\p L_1}{\p U^i} - \frac{d}{dt}\frac{\p L_1}{\p A^i}  \qquad \mbox{and} \qquad
\mR_i = \frac{\p L_1}{\p A^i}.\]
It turns out that 
\begin{align}
\begin{split}
\mP &= - u^{-2} \left( A' - u^{-2}|A|^2 U -2 u^{-2} \langle A,U\rangle A + 4 u^{-4} \langle A,U \rangle^2 U - 2 u^{-2} \langle A',U \rangle U \right) \\
&= -C, \label{Pdefine} 
\end{split}\\
\mR &=u^{-2} A - 2 u^{-4} \langle U,A \rangle U. \label{Rdefine}
\end{align}
The Hamiltonian is then given by the Legendre transform ${H = \langle \mP, X' \rangle + \langle \mR, U' \rangle - L_1,}$ which gives
\be \label{H} H  \; = \; \langle \mP, U \rangle  - \langle \mR, U \rangle^2 + \frac{1}{2} u^{2} |\mR|^2.\ee
It follows that the conformal Mercator equation \eqref{confMeqflat} is equivalent to the $4n$ first order Hamilton's equations
\[{X^i}' = \frac{\p H}{\p \mP_i}, \qquad  {U^i}' = \frac{\p H}{\p \mR_i}, \qquad  {\mP_i}' = - \frac{\p H}{\p X^i}, \qquad   {\mR_i}' = - \frac{\p H}{\p U^i},\]
that is, \cite{DK21}
\be \label{Hamilsys} X' = U, \quad U' = u^2 {\mc R} - 2\langle U,\mc R \rangle U, \quad {\mc P}' = 0, \quad {\mc R}' = -|\mR|^2 U + 2 \langle  U, \mR \rangle \mR - \mP. \ee

Let us recall (see e.g. \cite{A89}) that given 
 a function $f(X, U, {\mc P}, {\mc R})$ on the Hamiltonian phase space, the derivative of $f$ along a solution curve of \eqref{Hamilsys} is given by
 \[\frac{df}{dt} = \{f, H\}, \]
 where for any two functions $f, g$ on the phase space, the Poisson bracket $\{ \,, \, \}$  is defined by
\[\{f, g\} = \frac{\p f}{\p X^i} \frac{\p g}{\p \mP_i} + \frac{\p f}{\p U^i} \frac{\p g}{\p \mR_i} - \frac{\p g}{\p X^i} \frac{\p f}{\p \mP_i} - \frac{\p g}{\p U^i} \frac{\p f}{\p \mR_i},\]
using the summation convention.  

It then follows that a function $f(X, U, {\mc P}, {\mc R})$ is constant along a solution curve if and only if 
$\{f, H\} = 0.$


\subsection{Conserved quantities in phase-space variables}

It turns out that the conserved quantities for solutions of the conformal Mercator equation, \eqref{FT} - \eqref{FS}, are all polynomials in the phase space variables.  Treating $X, U, {\mc P}, {\mc R}$ as $n$-vectors, the quantities are given by
\begin{eqnarray*}
F_T &=& \langle T, \mP \rangle, \\
F_R &=& R_{ij} (X^i \mP^j + U^i \mR^j),\\
F_D &=& a \, ( \langle X, \mP \rangle + \langle U, \mR\rangle ), \\
F_S &=& \langle S, |X|^2 \mP + 2 \langle X, U \rangle \mR -2(\langle X, \mP \rangle +\langle U, \mR \rangle) X - 2(1+\langle X,\mR \rangle) U \rangle,
\end{eqnarray*}
where in obtaining these expressions we have used the conformal Mercator equation in Hamilton's form \eqref{Hamilsys}, and also the relations \eqref{Pdefine}, \eqref{Rdefine}, which imply 
\begin{eqnarray*} 
A &=& -2 \langle \mR, U \rangle U + u^2 \mR, \\
A' &=& \left(2 \langle U, \mP \rangle + 4 \langle U, \mR \rangle^2 - u^2|\mR|^2\right) U - 2 u^2 \langle U, \mR \rangle \mR  - u^2 \mP.
\end{eqnarray*}

In writing down relations between the above conserved quantities and those of Proposition \ref{prop_conservedqtractor}, we shall use the following set of $(n+1)(n+2)/2$ quantities, which correspond to a choice of $(n+1)(n+2)/2$ basis vector fields of the Lie algebra of the conformal Killing vector fields:
\begin{eqnarray}
E_T^i &=& \mP^i \label{F_T}\\
E_{R}^{ij} &=&  \ep^{ij}(X, \mP) + \ep^{ij}(U, \mR) \qquad  \quad \label{F_R}\\
E_{D} &=& \langle X, \mP \rangle + \langle U, \mR\rangle \label{F_D}\\
E_S^i &=& |X|^2 \mP^i + 2 \langle X, U\rangle \mR^i -2( \langle X, \mP \rangle + \langle U, \mR\rangle ) X^i - 2(1+ \langle X,\mR \rangle) U^i, \qquad \quad  \label{F_S}
\end{eqnarray}
where $i,j=1, \dots, n,$
and $\ep^{ij}$ is the elementary alternating tensor of rank $2,$ defined in Section \ref{sectract}.  

On the other hand, 
the conserved quantities \eqref{Q0ijN} - \eqref{Qijkp} of Proposition \ref{prop_conservedqtractor} can be written as 
\begin{align}
 Q^{0ijN} &= -\langle \mR, U \rangle \ep^{ij}(U,\mR) + \ep^{ij}(U,\mP) -    \ep^{ijk}(U,\mR, \mP) X_k \label{Q1}\\
 \begin{split}
Q^{0ijk} &= - \Big( \, \langle U,\mR \rangle \ep^{ijk}(X,U,\mR) +  \ep^{ijk}(U, X,\mP)
\\ &  \qquad \quad + \frac{1}{2}|X|^2 \ep^{ijk}(U,\mR, \mP) + \ep^{ijkl}(X,U,\mR, \mP) X_l \, \Big) \quad \qquad \label{Q2} 
\end{split}
\\
Q^{ijkN} &= - \ep^{ijk}(U,\mR,\mP) \label{Q3} \\
Q^{ijkl} &= - \ep^{ijkl}(X, U,\mR,\mP). \label{Q4}
\end{align}

\subsection{Relation between conserved quantities}

Recall that the quantities of Proposition \ref{prop_conservedqtractor} are constant along any curve satisfying the tractor condition $\Delta_5 = 0$ and $\Delta_4$ non-zero constant, while the quantities of Proposition \ref{prop_firstintegralF}, and subsequently Proposition \ref{prop_FTFRFDFS}, are constant along all solutions of the conformal Mercator equation \eqref{confMeqflat}.  The two classes of (parametrized) curves satisfy different equations, thus they are not expected to have related conserved quantities.  
Nevertheless, we have seen that along the logarithmic spiral \eqref{logspiral} the two sets of quantities are both constant.  Our main result is Theorem \ref{prop_relation} below, which gives explicit relations between the quantities.

\begin{theo} \label{prop_relation}
The conserved quantities of  Proposition {\rm\ref{prop_conservedqtractor}}, \eqref{Q1} - \eqref{Q4}, can be written 
in terms of the quantities \eqref{F_T} - \eqref{F_S} as follows.
Let $E_{T} := (\, E_T^i \, )$  and 
$E_{S} := (\, E_S^i \, )$  denote the $n$-vectors whose components are $E_T^i $ and $E_S^i ,$ respectively.  Then
\begin{eqnarray}
Q^{0ijN} &=& \frac{1}{2} \ep^{ij}(E_{T}, E_{S}) - E_{D} \, E_{R}^{ij} \label{Q1F} \\
Q^{0ijk} &=&  \frac{3}{2} E_{R}^{[ij}  \, E_{S}^{k]} \label{Q2F} \\
Q^{ijkN} &=& - 3 \, E_{R}^{[ij}  \, E_{T}^{k]} \label{Q3F} \\
Q^{ijkl} &=& \frac{6}{E_{D}} \, E_{R}^{[ij}  \, E_{S}^{k} \, E_{T}^{l]} . \label{Q4F}
\end{eqnarray}
\end{theo}
\begin{proof}
The proof is by direct calculation, where we use identities from the property of the determinant such as $\ep^{ijk}(U, \mR, \mP) = 3 \, \ep^{[ij}(U, \mR) \, \mP^{k]}.$
Note also that $Q^{0ijk}$ \eqref{Q2} can be written as
\begin{eqnarray*}
Q^{0ijk}   &=& \frac{1}{2}|X|^2 \ep^{ijk}(U, \mR, \mP)  + \langle X,U\rangle \ep^{ijk}(X, \mP,\mR) - (1+\langle X, \mR\rangle) \ep^{ijk}(X, \mP, U) \\
&{}& - (\langle X, \mP\rangle + \langle U, \mR\rangle ) \ep^{ijk}(U, \mR, X),
\end{eqnarray*}	 
using the equality 
\begin{eqnarray*}
\ep^{ijkl}(X,U,\mR, \mP) X_l  &=& - |X|^2 \ep^{ijk}(U, \mR, \mP)  + \langle X,U\rangle \ep^{ijk}(X, \mR, \mP) - \langle X, \mR\rangle \ep^{ijk}(X, U, \mP) \\
&{}& + \langle X, \mP\rangle \ep^{ijk}(X, U, \mR),
\end{eqnarray*}	
which also follows from the property of the determinant.
\end{proof}
Note that although being proved here in the Hamiltonian formulation, the identities \eqref{Q1F} - \eqref{Q4F} of 
Theorem \ref{prop_relation} hold regardless of whether the quantities are expressed in the phase space or `spacetime' variables.

\begin{exam}
Here we come back to the logarithmic spiral \eqref{logspiral},
 \[ \wh X(t) = e^t  \cos(ct) P_0 +  e^t  \sin(ct) Q_0 + R_0.\]
Writing the quantities \eqref{F_T} - \eqref{F_S} back in the spacetime (configulation space) variables, it can be shown that 
 \[E_T^i = 0, E_R^{ij} = \frac{c}{|P_0|^2}\ep^{ij}(P_0,Q_0), \quad E_D = -1, \quad E_S^i = -2c \frac{\langle Q_0, R_0 \rangle}{|P_0|^2}P_0^i + 2c \frac{\langle P_0, R_0 \rangle}{|P_0|^2}Q_0^i+ 2R_0^i. \]
 
Then one can use Theorem \ref{prop_relation} to compute the conserved quantities of Proposition \ref{prop_conservedqtractor}.
Since $E_T^i = 0,$ we have that $Q^{ijkN} =Q^{ijkl} =0.$  Then
\[Q^{0ijN} =- E_{D} \, E_{R}^{ij} = \frac{c}{|P_0|^2}\ep^{ij}(P_0,Q_0),\]
and
\[Q^{0ijk} =  \frac{3}{2} E_{R}^{[ij}  \, E_{S}^{k]} = 3 \frac{c}{|P_0|^2}\ep^{[ij}(P_0,Q_0) \, R_0^{k]} = \frac{c}{|P_0|^2}\ep^{ijk}(P_0,Q_0, R_0).\]
Cf. Example \ref{exam_logspiral}.
\end{exam}

\begin{exam}  
In $3$ dimensions, there are $3$ quantities for $Q^{0ij4};$ let $Q_1$ denote a $3$-vector given by $Q_1 = (Q^{0234}, - Q^{0134}, Q^{0124}).$  There is only one quantity for each of $Q^{0ijk}$ and $Q^{ijk4}$, let $Q_2:=Q^{0123}$ and $Q_3:=Q^{1234}.$  Also recall that $Q^{ijkl}$ does not appear in $3$ dimensions.

The relations \eqref{Q1F} - \eqref{Q3F} of Theorem \ref{prop_relation} simplify, as there are also $3$ quantities for $E_{R}^{ij},$ and thus one can form a $3$-vector $E_{R} := (E_R^{23}, - E_R^{13}, E_R^{12}).$   Then $Q_1,$ $Q_2$ and $Q_3$ can be given simply in terms of the dot products $\langle \, , \, \rangle$ and cross products $\w$ of the $3$-vectors  $E_{T},$ $E_{R},$ $E_{S},$ and a scalar multiple with $E_{D}$ as
\begin{eqnarray*}
Q_{1} &=&  \frac{1}{2} E_T \w E_S - E_D E_R,  \\
Q_{2} &=& \frac{1}{2}\langle E_R, E_S \rangle,\\
Q_3 &=& - \langle E_T, E_R \rangle.
\end{eqnarray*}

In addition, the Hamiltonian \eqref{H},
which is also a conserved quantity can be written as 
\[H = \frac{1}{2}\left( |E_R|^2 - \langle E_T, E_S \rangle - {(E_D)}^2 \right).\]
\end{exam}

\subsection{Observations on Poisson bracket relations and Hamiltonian phase flows}
We end this section with some observations on the Poisson bracket relations, and give an example of a set of functionally independent conserved quantities in involution.  In addition, we describe the Hamiltonian phase flow with one of our conserved quantities from tractors, and give an example of how the flow acts on two solutions in $3$ dimensions.

First, recall that given a function $f$ on a symplectic manifold, with the symplectic structure $\om,$ the {\it Hamiltonian vector field} of $f,$ denoted by $V_f,$ is defined by ${df(K) = \om(K, V_f)}$ for any vector field $K.$  In the case of our $4n$-dimensional phase space, the Hamiltonian vector field $V_f$ is given in the canonical coordinates $(X,U,{\mc P}, {\mc R})$ by 
\[V_f =  \frac{\p f}{\p \mP_i} \frac{\p }{\p X^i} + \frac{\p f}{\p \mR_i}  \frac{\p }{\p U^i} - \frac{\p f}{\p X^i} \frac{\p}{\p \mP_i} - \frac{\p f}{\p U^i} \frac{\p }{\p \mR_i}.\]
The integral curves of a Hamiltonian vector field $V_f$ give rise to a $1$-parameter group of transformations on the phase space, called the {\it Hamiltonian phase flow with $f.$}

We shall follow the convention in \cite{A89} in defining the Poisson bracket $\{V_f, V_g\}$ of Hamiltonian vector fields $V_f,$$V_g$ to be  $\{V_f, V_g\} = - [V_f, V_g],$ where $[V_f, V_g]$ denotes the usual Lie bracket or commutator of vector fields as derivations on a manifold, that is, $[V_f, V_g](h) = V_f(V_g(h)) -  V_g(V_f(h))$  for any function $h.$  With this convention, it can be shown that $\{V_f, V_g\}$ is the Hamiltonian vector field of $\{f,g\}.$   Given a Hamiltonian (in our case the Hamiltonian \eqref{H}), it can be shown that all conserved quantities or first integrals of the Hamilton's equations form an algebra under the Poisson bracket (a subalgebra of the algebra of all functions on the manifold), with the corresponding algebra of the Hamiltonian vector fields. (See e.g. \cite{A89}.)    Within this, moreover, some subset of the conserved quantities may also form a subalgebra.

Now, consider the quantities $E_T^i,$ $E_{R}^{ij},$ $E_{D},$ $E_S^i,$ \eqref{F_T} - \eqref{F_S}.  Recall that they are respectively constructed from the basis of the Lie algebra of conformal Killing vector fields on the configuration space $\R^n:$
\be \label{basisconfvf} \frac{\p}{\p x^i}, \quad  x_i \frac{\p}{\p x^j} - x_j \frac{\p}{\p x^i}, \quad x^k \frac{\p}{\p x^k}, \quad  x_k x^k  \frac{\p}{\p x^i} - 2 x_i x^k \frac{\p}{\p x^k}.\ee 
By direct calculation, it can be shown that the Poisson bracket relations among $E_T^i,$ $E_{R}^{ij},$ $E_{D}$ and $E_S^i$ are given by 
\[\{E_T^i, E_T^j\} = 0, \quad \{E_T^i, E_R^{jk}\} = - \de^{ij} E_T^k + \de^{ik} E_T^j, \quad \{E_T^i, E_D\} = -E_T^i, \]
\[ \{E_T^i, E_S^{j}\} = - 2 E_R^{ij}  + 2 \de^{ij} E_D, \quad  \{E_R^{ij}, E_R^{kl}\} = \de^{ik} E_R^{jl} + \de^{jl} E_R^{ik}- \de^{il} E_R^{jk} - \de^{jk} E_R^{il},  \]
\[\{E_R^{ij}, E_D\} = 0, \quad  \{E_R^{ij}, E_S^k\} =  \de^{ik} E_S^j - \de^{kj} E_S^i, \quad \{E_D, E_S^{i}\} = - E_S^i, \quad \{E_S^i, E_S^j\} = 0.  \]
 This shows that the vector space spanned by $\{E_T^i, E_{R}^{ij},E_{D}, E_S^i\}$ is an algebra under the Poisson bracket.  Moreover, the Poisson bracket relations follow (modulo an overall minus sign) the Lie bracket relations of the corresponding  generators \eqref{basisconfvf} of conformal transformations.  The Hamiltonian vector fields $V_{E_T^i},$ $V_{E_{R}^{ij}},$ $V_{E_{D}},$ $V_{E_S^i},$ which obey the same Poisson bracket relations as $E_T^i,$ $E_{R}^{ij},$ $E_{D},$ $E_S^i,$
are given in Appendix \ref{app_poisson}.  

On the other hand, for the conserved quantities from tractors $Q^{0ijN},$ $Q^{0ijk},$ $Q^{ijkN},$ $Q^{ijkl},$ \eqref{Q1} - \eqref{Q4}, the computation of the Poisson brackets gives complicated expressions and we are unable to recognise the Poisson bracket of a pair of them as a linear combination among themselves.   However, a symbolic calculation using Maple for the case of the $3$-dimensional configuration $(n=3)$ shows that the set of $5$ Hamiltonian vector fields $\{V_{Q^{0124}}, V_{Q^{0134}}, V_{Q^{0234}}, V_{Q^{0123}}, V_{Q^{1234}}\}$ does not span a Lie algebra.  This confirms that the linear span of $\{Q^{0ijN}, Q^{0ijk}, Q^{ijkN}, Q^{ijkl}\}$ is not in general closed under the Poisson bracket. 

This is perhaps not unexpected. The quantities $Q^{0ijN},$ $Q^{0ijk},$ $Q^{ijkN},$ $Q^{ijkl}$ are constructed from a choice of basis elements of the tractor space $\wedge^4\R^{n+2},$ which correspond to conformal Killing-Yano $3$-forms. To the best of our knowledge, there is currently no known algebra of conformal Killing-Yano $3$-forms. There are studies of algebra structures of Killing-Yano and conformal Killing-Yano forms in literature (see e.g. \cite{KRT07, E16}), but the brackets defined involve forms of different degrees.  For example, it is shown in \cite{E16} that conformal Killing-Yano forms on constant curvature manifolds form a graded Lie algebra under a modified Schouten-Nijenhuis bracket, which maps a $p$-form and a $q$-form to a $p+q-1$-form.  

It would be interesting to explore further the Poisson bracket of each pair in the set $\{Q^{0ijN}, Q^{0ijk}, Q^{ijkN}, Q^{ijkl}\},$ and to determine whether it would be possible to construct a larger set, consisting of the  quantities and their Poisson brackets, that is closed under the bracket.  Such a study might lead to an algebra structure involving conformal Killing-Yano forms.

\vspace{0.2cm}

Next, we look into how the quantities $Q^{0ijN},$ $Q^{0ijk},$ $Q^{ijkN},$ $Q^{ijkl},$ \eqref{Q1} - \eqref{Q4}, might be complimentary to $E_T^i,$ $E_{R}^{ij},$ $E_{D},$ $E_S^i,$ \eqref{F_T} - \eqref{F_S}, in terms of forming a set of involutive functionally independent conserved quantities, in order to confirm Liouville integrability of the conformal Mercator equation \eqref{Hamilsys}.

First, recall that Liouville integrability of a system of Hamilton's equations on a $4n$-dimensional phase space is guaranteed if there exist $2n$ functionally independent conserved quantities that are in involution, where two functions $f, g$ are said to be involutive if $\{f, g\}=0.$  

Let us first remark on functional dependence.  There cannot be more than $4n$ functionally independent functions on a $4n$-dimensional phase space.   The number of quantities $E_T^i,$ $E_{R}^{ij},$ $E_{D},$ $E_S^i$ is in total $\frac{1}{2}(n+1)(n+2),$ while that of  $Q^{0ijN},$ $Q^{0ijk},$ $Q^{ijkN},$ $Q^{ijkl}$ is ${\frac{1}{24}n(n-1)(n+1)(n+2).}$  Both numbers are greater than $4n$ for $n \ge 5.$  Notice that the set $\{E_T^i = \mP^i\}$ is always functionally independent.  

We explore the concrete example of $n=3.$  The phase space is of $12$ dimensions, and one needs a set of $6$ involutive functionally independent conserved quantities to confirm Liouville integrability.  In fact, the set $\{E_T^1, E_T^2, E_T^3, E_{R}^{12}, E_{R}^{13},E_{R}^{23},E_{D}, E_S^1,E_S^2, E_S^3\}$ is functionally independent, and so is the set  $\{Q^{0124},Q^{0134},Q^{0234}, Q^{0123}, Q^{1234}\}.$ 
However, most of the Poisson brackets among them are nonzero.  The largest involutive independent set we are able to find among all above quantities, plus the Hamiltonian $H$ \eqref{H}, consists of $5$ quantities, namely $\{E_{T}^i, Q^{1234}, H\},$ $i=1,2,3.$   

An observation in this regard is that although we find that $\{E_T^l, Q^{ijkN}\} = 0$ in any dimension $n$, the $\frac{1}{6}n(n-1)(n-2)$ \, $Q^{ijkN}$ quantities do not commute among themselves, and thus only one of them can contribute to the involutive set.  Therefore the largest involutive independent set we are currently able to find consists of $n+2$ quantities. 

\vspace{0.2cm}

Lastly, let us give a remark on the Hamiltonian phase flows with our conserved quantities.  Recall that two Hamiltonian phase flows commute if and only if the Poisson bracket of the corresponding functions is (locally) constant.  (See e.g. \cite{A89}.)  This implies that the phase flow with a conserved quantity is a one-parameter group of symmetry transformations of the Hamilton's equations, i.e. they map solutions to solutions.  

On the configuration space, it is known that conformal transformations are symmetries of the conformal Mercator equation \eqref{confMeqflat}, and this leads to our construction of the conserved quantities $E_T^i,$ $E_{R}^{ij},$ $E_{D},$ $E_S^i,$ \eqref{F_T} - \eqref{F_S}.  Their Hamiltonian vector fields, given in Appendix \ref{app_poisson}, are extensions of the corresponding generators of conformal transformations \eqref{basisconfvf} to the phase space (cf. the exact same terms on the $\frac{\p}{\p X^i}$ coordinate bases), and these generate symmetry transformations of  the conformal Mercator equation \eqref{Hamilsys} on the phase space.

Similarly, the Hamiltonian phase flows with the conserved quantities $Q^{0ijN},$ $Q^{0ijk},$ $Q^{ijkN},$ $Q^{ijkl},$ \eqref{Q1} - \eqref{Q4} are also symmetry transformations of \eqref{Hamilsys}.  The difference here is that they do not a priori correspond to one-parameter groups of transformations on the configuration space.  Since $Q^{0ijN},$ $Q^{0ijk},$ $Q^{ijkN},$ $Q^{ijkl}$ are constructed from tractors which correspond to conformal Killing-Yano $3$-forms, it would be interesting to understand how their Hamiltonian phase flows can be interpreted on the configuration space in relation to the conformal Killing-Yano $3$-forms.  Due to the complicated forms of the Hamiltonian vector fields of $Q^{0ijN},$ $Q^{0ijk},$ $Q^{ijkN},$ $Q^{ijkl},$ given in Appendix \ref{app_poisson}, we are yet to make much progress in this direction.  Let us however end by discussing the simplest of the flows, namely that of $Q^{ijkN} = -\ep^{ijk}(U,\mR, \mP).$

The  Hamiltonian vector field $V_{Q^{ijkN}}$ is given by 
\be \label{VQijkN} V_{Q^{ijkN}} = \ep^{ijk}(\mR,U, \frac{\p}{\p X}) + \ep^{ijk}(U, \mP, \frac{\p}{\p U}) + \ep^{ijk}(\mR,\mP, \frac{\p}{\p \mR}), \ee
where $\frac{\p}{\p X}$ denotes the $n$-vector $(\frac{\p}{\p X^m}),$ and similarly for $\frac{\p}{\p U}$  and $\frac{\p}{\p \mR}.$  The flow with $Q^{ijkN}$ is then obtained by solving for the integral curves of $V_{Q^{ijkN}},$ which amounts to solving a system of $4n$ ODEs.  However, $Q^{ijkN}$ is independent of $X$ and thus $V_{Q^{ijkN}}$ has no $\frac{\p}{\p \mP^i}$ components. This implies that the vector $\mP$ is constant along an integral curve.

With $\mP$ being constant, and assuming that $\mP^i, \mP^j, \mP^k$ are not simultaneously zero,  the second term of \eqref{VQijkN} 
shows that the flow with $Q^{ijkN}$ acts on the $3$-dimensional subspace spanned by $\{\frac{\p}{\p U^i}, \frac{\p}{\p U^j}, \frac{\p}{\p U^k}\}$ as rotations on planes perpendicular to the vector $(\mP^i, \mP^j, \mP^k).$ 
Analogously, the third term of \eqref{VQijkN} implies that the flow also acts on the $3$-dimensional subspace spanned by $\{\frac{\p}{\p \mR^i}, \frac{\p}{\p \mR^j}, \frac{\p}{\p \mR^k}\}$ as rotations on planes perpendicular to the vector $(\mP^i, \mP^j, \mP^k).$  
Having solved for $U$ and $\mR,$ they can be substituted into the first term of \eqref{VQijkN} which gives a system of ODEs  for $X.$  

Let $\mP^{ijk}$ denote the  $3$-vector $\mP^{ijk}:=(P^i, P^j, P^k),$ and similarly for $X^{ijk},$ $U^{ijk}$ and $\mR^{ijk}.$  Assume that $\mP^{ijk}$ is nonzero.  
One finds that the flow with $Q^{ijkN}$ is given by
 \begin{align}
\wt X^{ijk} &= X^{ijk} - \frac{\langle U^{ijk} \w \mR^{ijk}, \mP^{ijk} \rangle}{|\mP^{ijk}|^2} \mP^{ijk} s + \frac{(U^{ijk} \w \mR^{ijk}) \w \mP^{ijk}}{|\mP^{ijk}|^2} \Big(\cos\big(|\mP^{ijk}|s \big) -1\Big)  \nonumber \\
&{ } \hspace{1.4cm}  + \left( \frac{\langle U^{ijk} \w \mR^{ijk}, \mP^{ijk} \rangle}{|\mP^{ijk}|^3} \mP^{ijk} - \frac{U^{ijk} \w \mR^{ijk}}{|\mP^{ijk}|}\right) \sin\big(|\mP^{ijk}| s\big) \label{flowQijkNX}\\[0.2cm]
\wt X^m &= X^m \quad \mbox{for all} \quad  m \ne i, j, k \nonumber
 \end{align}
  \begin{align}
\wt U^{ijk} &=  U^{ijk} \cos\big(|\mP^{ijk}| s \big) +  \frac{\langle U^{ijk}, \mP^{ijk} \rangle}{|\mP^{ijk}|^2}  \mP^{ijk} \Big(1-\cos\big(|\mP^{ijk}| s \big) \Big) + \frac{U^{ijk} \w \mP^{ijk}}{|\mP^{ijk}|} \sin\big(|\mP^{ijk}| s\big)\nonumber \\[0.2cm]
\wt U^m &= U^m \quad \mbox{for all} \quad  m \ne i, j, k \nonumber
\\[0.2cm]
\wt \mP^m &= \mP^m \quad \mbox{for all} \quad   m = 1, \dots, n\nonumber
\\[0.2cm]
\wt \mR^{ijk} &=  \mR^{ijk} \cos\big(|\mP^{ijk}| s \big) +  \frac{\langle \mR^{ijk}, \mP^{ijk} \rangle}{|\mP^{ijk}|^2}  \mP^{ijk} \Big(1-\cos\big(|\mP^{ijk}| s \big) \Big) + \frac{\mR^{ijk} \w \mP^{ijk}}{|\mP^{ijk}|} \sin\big(|\mP^{ijk}| s\big) \nonumber \\[0.2cm] 
\wt \mR^m &= \mR^m \quad \mbox{for all} \quad  m \ne i, j, k, \nonumber
 \end{align}
where $s$ is the parameter of the flow, and $\langle \, , \, \rangle$ and $\w$ denote the dot product and cross product of $3$-vectors, respectively.
 
Note that if $\mP^{ijk}$ is the zero vector, the second and third terms of \eqref{VQijkN} vanish and the flow with $Q^{ijkN}$ is given by
 \begin{align}
 \begin{split}
\wt X^{ijk} &= X^{ijk} + s   \, \mR^{ijk} \w U^{ijk},  \quad 
\wt X^m = X^m  \quad   \mbox{for all} \quad  m \ne i, j, k 
\\[0.2cm]
 \mbox{and} \quad \wt U^m &= U^m, \quad \wt \mP^m = \mP^m,  \quad \wt \mR^m = \mR^m \quad \mbox{for all} \quad  m = 1, \dots, n. \label{flowQijkNP0}
\end{split}
 \end{align}
This is the translation on the $3$-dimensional subspace spanned by $\{\frac{\p}{\p X^i}, \frac{\p}{\p X^j}, \frac{\p}{\p X^k}\}$ along the direction of the $3$-vector $\mR^{ijk} \w U^{ijk}.$

\begin{exam}
For the logarithmic spiral \eqref{logspiral},
 \[ \wh X(t) = e^t  \cos(ct) P_0 +  e^t  \sin(ct) Q_0 + R_0,\]
since $\mP = 0$ for this solution, the action of the flow with $Q^{ijkN}$ on the solution  is given by the translation \eqref{flowQijkNP0}.

For concreteness, let us consider the case of the $3$-dimensional configuration space, and let $P_0 = (1,0,0),$ $Q_0 = (0,1,0)$ and $R_0$ be the origin.  Then on the $12$-dimensional phase space, the $X$ coordinates of this solution can be written as the $3$-vector by
\be \label{examflowlogspiral} X(t) = (e^t  \cos(ct), e^t  \sin(ct), 0). \ee
Now, $U = X'$ and $\mR$ defined by \eqref{Rdefine} are given by
\[U(t) = \Big(e^t  \big(\cos(ct) - c \sin(ct) \big), e^t  \big(\sin(ct) + c \cos(ct) \big), 0 \Big), \; \mR(t) = -(e^{-t}  \cos(ct), e^{-t}  \sin(ct), 0),  \]
and $\mR \w U = (0,0,-c).$  Hence, the flow with $Q^{1234}$ maps solution \eqref{examflowlogspiral} to solution 
\be \label{examflowlogspiraltransf} \wt X(t) = (e^t  \cos(ct), e^t  \sin(ct), -cs), \ee
which is a family of logarithmic spirals.

Lastly, let us consider the solution
\be \label{examflowStransflogspiral} X_{sp}(t) = \left(\frac{- e^{2t} +e^t\cos t}{1-2e^t \cos t + e^{2t}}, \, \frac{e^t\sin t}{1-2e^t \cos t + e^{2t}}, \, 0 \right), \ee
which is obtained from the logarithmic spiral \eqref{examflowlogspiral}, setting $c=1,$ by a special conformal transformation on the configuration space.  For this solution, $\mP = (2, -2, 0).$  Hence, under the flow with $Q^{1234},$ $X_{sp}$ is transformed according to \eqref{flowQijkNX}, and this gives another solution \\$\wt  X_{sp}(t) = \left( \wt  X_{sp}^1(t), \, \wt X_{sp}^2(t), \, \wt  X_{sp}^3(t) \right)$ given by
 \begin{align*}
\wt  X_{sp}^1(t) &= \frac{  - (3+\cos s)  e^{2t} + 4 e^t \cos t - 2(1 - \cos s) e^t \sin t  + \cos s -1}{4(1-2e^t \cos t + e^{2t})}
\\[0.2cm]
\wt  X_{sp}^2(t) &= \frac{(1 -\cos s)  e^{2t} + 2(1 + \cos s) e^t \sin t  + \cos s -1}{4(1-2e^t \cos t + e^{2t})}
\\[0.2cm]
\wt  X_{sp}^3(t) &= \frac{\sqrt{2}\sin s ( e^{2t} - 2e^t \sin t -1)}{4(1-2e^t \cos t + e^{2t})},
 \end{align*}
where for convenience we have rescaled the parameter of the flow from $2\sqrt{2} \, s$ to $s.$
\end{exam}


\section*{Acknowledgements}
The authors would like to thank Maciej Dunajski for introducing them to the problem.
They are also grateful to Maciej Dunajski, Jan Gregorovi\v c, Rod Gover,  Wojciech Kry\'nski, Josef \v Silhan and Vojt\v ech \v Z\' adn\' ik for helpful discussions. 
The authors would also like to thank the anonymous referees for stimulating comments, which lead to an investigation of the Poisson bracket relations and Hamiltonian phase flows.  They are also grateful to the Isaac Newton Institute for Mathematical Sciences, Cambridge, for support and hospitality during the programme Twistor theory, where some research on this paper was carried over. This work
was supported by EPSRC grant EP/Z000580/1.
L.Z. is supported by the grant GACR 24-10887S Cartan supergeometries and Higher Cartan geometries.


\section*{Declarations}
The authors declare no competing interests.

\appendix
\section{Conserved quantities for conformal circles from parallel $3$-tractor} \label{A1}
Let us give here for completeness the quantities for conformal circles coming from the pairing of the $3$-tractor $\mathbb{T} \wedge \mathbb{U} \wedge \mathbb{A}$ and parallel tractors corresponding to cKY 2-forms, aka elements of $\wedge^3 \R^{n+2}$. 
Computations analogous to Section \ref{sectract} give that the $3$-tractor takes the form
$$
T_3= u^{-3} U^{[i}A^{j]} \: e_{0ij}-u^{-1} U^i\: e_{0iN}
$$
where we use Einstein summation convention,  $[ab..c]$ denotes the anti-symmetrization over indices in the bracket and $i,j=1,\dots, n$.
One can see that this corresponds to the results presented in \cite[Section 4.4.3.]{GST21}. 
Using notation analogous to Section \ref{sectract}, we get the quantities as follows:
 \begin{align*}
Q^{0iN} &= 
u^{-1}U^i+u^{-3}  \epsilon^{ik}(U,A) X_k
 \\
Q^{0ij}& =
 -u^{-1}\epsilon^{ij}(X,U)+\frac{1}{2} u^{-3} |X|^2 \ep^{ij}(U, A) - u^{-3} \ep^{ijk}(X,U,A) X_k
\\
 Q^{0ij}&=u^{-3} \epsilon^{ij}(U,A)
\\
 Q^{ijk}&=u^{-3}\epsilon^{ijk}(X,U,A). 
\end{align*}


\section{Conserved quantities for conformal circles from Lagrangian formulation} \label{app_circleV} 
It was shown in \cite{DK21} that conformal circles correspond to the class of solutions of the conformal Mercator equation \eqref{confMeqflat} satisfying an additional condition
\[C= \left(\frac{1}{2} u^{-4}|A|^2 - 2 u^{-6} \langle A,U \rangle^2\right)U + u^{-4} \langle A,U \rangle A.\]
This reduces \eqref{confMeqflat} to the conformal circle equation
\be \label{confgeqflat} A' -3 u^{-2} \langle A,U \rangle A - \frac{3}{2}u^{-2}|A|^2 U= 0, \ee
the general solution of which is a projectively parametrized circle
\[X(t) = X_0 + \frac{t U_0 + t^2 A_0}{1+t^2 |A_0|^2},\]
where $X_0, U_0, A_0$ are constant vectors, $|U_0|=1$ and $\langle U_0, A_0 \rangle =0.$
  
It turns out that for conformal circles, the conserved quantity of Proposition \ref{prop_firstintegralF} takes a simple form.  To see this, first note that the function $F$ \eqref{F} can be written as
\[ F = \frac{d^2}{d t^2}\langle W, V \rangle - \langle Z, V \rangle, \quad \mbox{where} \quad
Z =  \frac{1}{2} u^{-2} A' + \frac{3}{2} W'' - \langle W, A' \rangle W.\]
Now, by virtue of \eqref{confgeqflat}, it can be shown that the vector field $Z$ vanishes on any conformal circle.  This gives 
\[ F = \frac{d^2}{d t^2}\langle W, V \rangle.\]
Recall that for all conformal Killing vector fields $V$ except for the generator of a special conformal transformation, we have that $\langle W, V' \rangle$ is constant, and thus $\dsl F = \frac{d}{dt} \langle W', V \rangle.$

\section{Note on tractor equation}  \label{app_tractoreq}

The general principle states that the middle, aka invariant, slot of the tractor equation \eqref{tract-comb-5}
 gives the equation for the corresponding curves satisfying $\Delta_5=0$, $\Delta_4 \neq 0$. This can be done in a general setting, but it will be nasty. Involving the condition $\Delta'_4=0$ (see \eqref{Delta} for $\Delta_4$) related to discussion in Proposition \ref{prop_conservedqtractor}, this equation simplifies to 
 \begin{align} 
 \begin{split}
 \label{tr}
&(12 u^{-7}\langle U, A \rangle^3 -8 u^{-5} \langle U,A \rangle \langle U, A' \rangle-3u^{-5}\langle U,A \rangle \langle A,A \rangle +u^{-3}\langle U,A'' \rangle)U+\\
&(-12u^{-5}\langle U,A\rangle^2+4u^{-3}\langle U, A' \rangle+3u^{-3}\langle A, A \rangle)A +4u^{-3}\langle U,A\rangle A' -u^{-1}A''=0.
\end{split}
 \end{align}
The Mercator equation  \eqref{confMeqflat} takes explicit form
\begin{align*}
&(-24u^{-8}\langle U,A\rangle^3 +16u^{-6}\langle U, A \rangle \langle U,A'\rangle +
12 u^{-6}\langle U, A \rangle \langle A, A \rangle -
 2u^{-4} \langle U, A''\rangle -4u^{-4}\langle A,A'\rangle)U \\ 
 &+(12 u^{-6}\langle U,A \rangle^2 -4u^{-4}\langle U,A'\rangle -3u^{-4}\langle A,A\rangle 
 ) A -4u^{-4}\langle U, A \rangle A'+u^{-2}A''=0.
\end{align*}
Clearly, these are different equations. However, involving the condition $\alpha'_1=0$ (see \eqref{alpha} for $\alpha_1$), the equation \eqref{tr} simplifies to
\begin{align} 
 \begin{split}
 \label{eq}
&(6 u^{-5} \langle U,A \rangle \langle U, A' \rangle -2u^{-3}\langle A,A'' \rangle)U+ 
\\& (-12u^{-5}\langle U,A\rangle^2+4u^{-3}\langle U, A' \rangle+3u^{-3}\langle A, A \rangle)A +4u^{-3}\langle U,A\rangle A' -u^{-1}A''=0.
\end{split}
 \end{align}
The condition $\alpha'_1=0$ relates to a special parametrisation of solution curves, so we can also try to apply it to the Mercator equation. 
It turns out that it then results in the minus $u^{-1}$-multiple of  \eqref{eq}. Thus the equations share solutions satisfying the conditions. Note that logarithmic spirals from Example \ref{exam_logspiral}
are examples of such solutions. The existence of multiple agrees with the fact that the tractor equation \emph{lives} in the space $TM[-1]$ while the flat Mercator equation \emph{lives} in the space $TM[-2]$. We note that this works only in the case of the conformal sphere. In the general curved case, the Mercator equation differs from the equation coming from tractor calculations by the curvature term $K$, \cite[Section 4.2.]{DK21}.


\section{Hamiltonian vector fields} \label{app_poisson}

The Hamiltonian vector fields of $E_T^i,$ $E_{R}^{ij},$ $E_{D},$ $E_S^i,$ \eqref{F_T} - \eqref{F_S}, and $Q^{0ijN},$ $Q^{0ijk},$ $Q^{ijkN},$ $Q^{ijkl},$ \eqref{Q1} - \eqref{Q4}, are given respectively as follows.
 \begin{align*}
V_{E_T^i} &= 
\frac{\p}{\p X^i}
 \\
 V_{E_R^{ij}} & =
X_i \frac{\p}{\p X^j} - X_j \frac{\p}{\p X^i} +  U_i \frac{\p}{\p U^j} - U_j \frac{\p}{\p U^i} +  \mP_i \frac{\p}{\p \mP^j} - \mP_j \frac{\p}{\p \mP^i} +  \mR_i \frac{\p}{\p \mR^j} - \mR_j \frac{\p}{\p \mR^i}
\\
V_{E_D} &=X^k \frac{\p}{\p X^k} +  U^k \frac{\p}{\p U^k} -  \mP_k \frac{\p}{\p \mP_k} -  \mR_k \frac{\p}{\p \mR_k}
\end{align*}
\begin{align*}
V_{E_S^i} &= |X|^2 \frac{\p}{\p X^i} - 2 X_i X^k \frac{\p}{\p X^k} \\
&+ 2 \left( \langle X, U \rangle \frac{\p}{\p U^i} - 2X_{(i}U_{k)} \frac{\p}{\p U_k} \right)
 \\
&+  2 \left( \big(\langle \mP, X \rangle + \langle \mR, U \rangle \big) \frac{\p}{\p \mP^i} - 2 \big(\mP_{[i}X_{k]} +\mR_{[i}U_{k]} \big) \frac{\p}{\p \mP_k} \right) \\
&+  2 \left( \big(1 + \langle \mR, X \rangle \big) \frac{\p}{\p \mR^i} - 2 \mR_{[i}X_{k]} \frac{\p}{\p \mR_k} \right) \\
V_{Q^{0ijN}} &= -\ep^{ij}(U, \mR) \left(U^m \frac{\p}{\p U^m} - \mR_m \frac{\p}{\p \mR_m}\right) - \langle U, \mR \rangle \left( \ep^{ij}(U, \frac{\p}{\p U}) + \ep^{ij}(\mR, \frac{\p}{\p \mR}) \right) \\
&+ \ep^{ij}(U, \frac{\p}{\p X}) + \ep^{ij}(\mP, \frac{\p}{\p \mR}) + X^k V_{Q^{ijkN}} + \ep^{ijm}(U, \mR, \mP) \frac{\p}{\p \mP^m}  \\
V_{Q^{0ijk}} &= \langle U, \mR \rangle \left( \ep^{ijk}(U, \mR,\frac{\p}{\p \mP}) + \ep^{ijk}(\mR, X, \frac{\p}{\p \mR})+ \ep^{ijk}(U, X,\frac{\p}{\p U})\right)\\
&-\ep^{ijk}(X, U, \mR) \left(U^m \frac{\p}{\p U^m} - \mR_m \frac{\p}{\p \mR_m}\right)
+ \ep^{ijk}(X, \mP,\frac{\p}{\p \mR}) + \ep^{ijk}(\mP, U, \frac{\p}{\p \mP})+ \ep^{ijk}(X, U, \frac{\p}{\p X})\\
&+  \ep^{ijk}(U, \mR,\mP) X^m \frac{\p}{\p \mP^m} + \frac{1}{2} |X|^2 V_{Q^{ijkN}} + X^l V_{Q^{ijkl}} +  \ep^{ijkl}(X, U, \mR,\mP) \frac{\p}{\p \mP^l} \\
V_{Q^{ijkN}} &= \ep^{ijk}(\mR,U, \frac{\p}{\p X}) + \ep^{ijk}(U, \mP, \frac{\p}{\p U}) + \ep^{ijk}(\mR,\mP, \frac{\p}{\p \mR})   \\
V_{Q^{ijkl}} &= -\ep^{ijkl}(X, U, \mR, \frac{\p}{\p X})  + \ep^{ijkl}(X, U, \mP, \frac{\p}{\p U}) - \ep^{ijkl}( U,  \mR, \mP,\frac{\p}{\p \mP}) + \ep^{ijkl}(X, \mR,  \mP, \frac{\p}{\p \mR}) \\
\end{align*}
where $\frac{\p}{\p X}$ denotes the $n$-vector $(\frac{\p}{\p X^m}),$ and similarly for $\frac{\p}{\p U},$ $\frac{\p}{\p \mP},$ and $\frac{\p}{\p \mR}.$




\begin{thebibliography}{99}

\bibitem{A89} Arnold, V. I. (1989) {\em Mathematical Methods of Classical Mechanics}, 2nd ed., Springer.

\bibitem{BE90}  Bailey, T. N. and Eastwood,  M. G. (1990) Conformal circles and parametrizations of curves in conformal manifolds, Proc. Amer. Math. Soc. {\bf 108}, 215–221.

\bibitem{BEG94}  Bailey, T. N., Eastwood,  M. G. and Gover, A.R. (1994) Thomas’s structure bundle for conformal, projective
and related structures, Rocky Mountain Jour. Math. {\bf 24}, 1191–1217.


\bibitem{CGH12} A. \v Cap, A.R. Gover, M. Hammerl, Normal BGG solutions and polynomials, Internat. J. Math. 23 No. 11 (2012) 1250117

\bibitem{CG17} Curry, S. N. and Gover, A. R. (2018) An introduction to conformal geometry and tractor calculus, with a view to applications in general relativity, In: Daud\'e, T., H\"afner, D. and Nicolas, J-P. (eds) {\em Asymptotic Analysis in General Relativity.} 86–170, Cambridge University Press.

\bibitem{DK21} Dunajski, M. and Kry\'nski, W.  (2021) Variational principles for conformal geodesics, Lett. Math. Phys. {\bf 111}(5), 127.

\bibitem{DT22} Dunajski, M. and Tod, K. P.  (2022) Conformal geodesics on gravitational instantons, Math. Proc. Camb. Philos. Soc. {\bf 173}(1), 123–154.

\bibitem{E23}  Eastwood, M. (2025) Conformal loxodromes, In: Pevzner, M., Sekiguchi, H. (eds) {\em Symmetry in Geometry and Analysis, Volume 1.} Progress in Mathematics. {\bf 357}, 251–268, Birkhäuser.

\bibitem{EZ22}  Eastwood,  M. G. and Zalabov\'a, L. (2022) Special metrics and scales in parabolic geometry, Ann. Global
Anal. Geom. {\bf 62}, 635–659.

\bibitem{E16} Ertem, \"U. (2016) Lie algebra of conformal Killing–Yano forms, Class. Quantum Grav. {\bf 33}, 125033.
 
\bibitem{GST21} Gover, A. R., Snell, D. and Taghavi-Chabert, A. (2021) Distinguished curves and integrability in Riemannian, conformal, and projective geometry,  Adv. Theor. Math. Phys. {\bf 25}(8), 2055–2118.


\bibitem{GZ}  Gregorovi\v c, J. and Zalabov\'a, L. (2023) First BGG operators on homogeneous conformal geometries,  Class. Quantum Grav. {\bf 40}, 065010.

\bibitem{KRT07} Kastor, D., Ray, S. and Traschen, J. (2007) Do Killing–Yano tensors form a Lie algebra?, Class. Quantum Grav. {\bf 24}, 3759.

\bibitem{K25} Kruglikov, B. (2025) Conformal geodesics are not variational in higher dimensions, arXiv:2505.06739v1 [gr-qc].

\bibitem{KMS24} Kruglikov, B., Matveev, V. S. and Steneker, W. (2025) Variationality of conformal geodesics in dimension $3,$  Anal. Math. Phys. {\bf 15}, 123.

\bibitem{M24} Marugame, T. (2024) The Fefferman metric for twistor CR manifolds and conformal geodesics in dimension three,  arXiv:2411.18961v3 [math.DG]. 

 \bibitem{O93} Olver, P. J. (1993) {\em Applications of Lie Groups to Differential Equations}, 2nd ed., Springer.

\bibitem{SZ19}  \v Silhan, J. and \v Z\'adn\'ik, V. (2019) Conformal theory of curves with tractors, J. Math. Anal. Appl. 
{\bf 473}, 112–140. 

\bibitem{Sulanke} Sulanke, R., Differential geometry of the Möbius space 1: curve theory, preprint available on \url{https://www2.mathematik.hu-berlin.de/~sulanke/Notebooks/emg.pdf}

\bibitem{T12}  Tod, K. P.  (2012) Some examples of the behaviour of conformal geodesics, J. Geom. Phys.
{\bf 62}, 1778–1792. 

\bibitem{Y57} Yano, K. (1957) {\em The Theory of Lie Derivatives and its Applications}, North-Holland Publishing Co.
\end{thebibliography}
\end{document}